\documentclass{article}

\usepackage[colorlinks,citecolor=blue,urlcolor=blue]{hyperref}

\usepackage[utf8]{inputenc}
\usepackage[english]{babel}

\usepackage{authblk}
\title{Rare Event Simulation for Steady-State Probabilities \\ via Recurrency Cycles\thanks{This article may be downloaded for personal use only. Any other use requires prior permission of the author and AIP Publishing. This article appeared in K. Bisewski et al., Chaos: An Interdisciplinary Journal of Nonlinear Science 29, no. 3 (2019): 033131. and may be found at \url{https://doi.org/10.1063/1.5080296}.}}
\author[1]{Krzysztof Bisewski\footnote{Email: bisewski@cwi.nl}}
\author[1,2]{Daan Crommelin}
\author[2]{Michel Mandjes}
\affil[1]{\small{Centrum Wiskunde \& Informatica, Amsterdam}}
\affil[2]{\small{Korteweg de Vries Institute for Mathematics, University of Amsterdam}}

\date{\today \vspace{-.5cm}}

\usepackage{natbib}

\usepackage{graphicx}
\usepackage{dcolumn}% Align table columns on decimal point
\usepackage{multirow} % allows for merging rows
\usepackage{epstopdf}
\usepackage{enumitem}
\usepackage{amsthm}
\usepackage{amsmath}
\numberwithin{equation}{section} % number equations with sections
\usepackage{amsfonts}
\usepackage{subcaption}
\usepackage{eurosym}
\usepackage{dsfont}
\usepackage{chngpage}
\usepackage{float}
\usepackage{listings}
\usepackage{color}
\usepackage{tikz}

\makeatletter
\newcommand*{\assenum}[1]{%
  \expandafter\@assenum\csname c@#1\endcsname%
}
\newcommand*{\assenumref}[1]{%
  \expandafter\@assenumref\csname c@#1\endcsname%
}
\newcommand*{\@assenum}[1]{%
  $\ifcase#1\or{\rm (I)}\or{\rm (II)}%\or{\rm (B)}\or{\rm (A^{**})}%
    \else\@ctrerr\fi$%
}
\newcommand*{\@assenumref}[1]{%
  $\ifcase#1\or{\rm I}\or{\rm II}%\or{\rm B}\or{\rm A^{**}}%
    \else\@ctrerr\fi$%
}
\AddEnumerateCounter{\assenum}{\@assenum}{1}
\AddEnumerateCounter{\assenumref}{\@assenumref}{1}
\makeatother

\DeclareMathOperator*{\argmax}{arg\,max}

\newtheorem{definition}{Definition}
\newtheorem{theorem}{Theorem}
\newtheorem{proposition}{Proposition}

\newtheorem{algorithm}{Algorithm}

\newcommand{\ind}{\mathds{1}}

\newcommand{\R}{\mathbb{R}}
\newcommand{\N}{\mathbb{N}}

\newcommand{\p}{\mathbb{P}}
\newcommand{\BRd}{\mathcal{B}(\R^d)}
\newcommand{\BpRd}{\mathcal{B}^+(\R^d)}
\newcommand{\eqd}{\stackrel{\mbox{\rm \tiny d}}{=}}
\renewcommand{\d}{{\rm d}}
\renewcommand{\a}{\alpha}
\renewcommand{\mid}{\,|\,}

\newcommand{\Exp}{\mathbb{E}}

\newcommand{\Var}{\mathbb{V}\textnormal{\textrm{ar}}}
\newcommand{\RE}{{\rm RE}}

\newcommand{\RTV}{{\rm RTV}}
\newcommand{\Eff}{{\rm Eff}}
\newcommand{\ep}{\varepsilon}

\newcommand{\upo}{^{(1)}}

\newcommand{\upot}{^{(1)}_t}
\newcommand{\uptt}{^{(2)}_t}

\newcommand{\qand}{\quad\text{and}\quad}

\newcommand{\dd}{\xrightarrow{\textnormal{d}}}
\newcommand{\red}[1]{{\color{black} #1}}
\newcommand{\redd}[1]{{\color{black} #1}}
\newcommand{\magenta}[1]{{\color{black} #1}}
\newcommand{\kb}[1]{{\tt \scriptsize KB: #1}}
\newcommand{\dc}[1]{{\color{black} #1}}
\newcommand{\vb}{\vspace{2mm}}
\renewcommand{\tilde}{\widetilde}
\renewcommand{\hat}{\widehat}
\renewcommand{\bar}{\overline}

\newcommand{\ass}{{\rm(\ref{ass:amrein}-\ref{ass:trare})}}

\begin{document}

\maketitle

\begin{abstract}
We develop a new algorithm for the estimation of rare event probabilities associated with the steady-state of a Markov stochastic process with continuous state space $\R^d$ and discrete time steps (i.e.\ a discrete-time $\R^d$-valued Markov chain). The algorithm, which we coin Recurrent Multilevel Splitting (RMS), relies on the Markov chain's underlying recurrent structure, in combination with the Multilevel Splitting method. Extensive simulation experiments are performed, including experiments with a nonlinear stochastic model that has some characteristics of complex climate models. The numerical experiments show that RMS can boost the computational efficiency by several orders of magnitude compared to the Monte Carlo method.
\end{abstract}

\section{Introduction}\label{s:introduction}

Many stochastic processes have a `stable regime', in the sense that with time their distribution converges to a so-called {\it steady-state}. The steady-state (or stationary, equilibrium, ergodic) probability distribution captures the long-term behavior of the process; the steady-state probability of an arbitrary event \dc{(or set)} $B$ is equal to the fraction of time the process spends in $B$ in the long run (irrespective of the process' initial value). In many application domains steady-state probabilities are of crucial interest; think of physics (e.g.\ particle systems), chemistry (e.g.\ reaction networks), and operations research (e.g.\ queueing systems). Within this context of steady-state distributions, an important subdomain concerns the analysis of {\it rare events}. Particularly when it concerns rare events with a potentially catastrophic impact, there is a clear need to accurately estimate their likelihood (earthquakes, extreme weather conditions, simultaneous failure of multiple components of a machine, etc.). As examples, we refer to \cite{ragone2018computation} for rare-event simulation methods in \magenta{the} climate context, and to  \cite{rubino2009rare} for a textbook treatment covering applications in e.g.\ engineering, chemistry, and biology.

\vb

Despite the evident importance of being able to estimate steady-state rare-event probabilities, relatively little attention has been paid to \magenta{the} development of efficient algorithms; rare-event simulation in a finite-time horizon context \magenta{received considerably more attention} (\magenta{focusing  e.g.\ on the estimation} of the probability to hit a set $B_1$ before hitting another set  $B_2$). The main contribution of this paper concerns the development of a broadly applicable rare-event simulation method that is tailored to the estimation of small steady-state probabilities. 

\vb

In our setup we focus on discrete-time $\R^d$-valued Markov chains. This framework covers a wide class of intensively used stochastic models. It for instance includes the numerical solutions to stochastic differential equations (SDEs), see e.g.\ \cite{kloeden1992numerical}. In addition, various (inherently discrete-time) standard models from e.g.\ finance, biology, and econometrics fall under this umbrella. The main advantage of our proposed algorithm is its broad applicability, the fact that it does not require detailed knowledge of the system under study, and that it is fairly straightforward to implement. In the sequel, we let $(X_n)_{n\in\N}$ be our $d$-dimensional Markov chain, which we assume to admit the stationary distribution $\mu$. We are interested in the probability that in steady-state the process attains a value in the set $B$, i.e.,
\begin{equation}\label{def:gamma}
\gamma := \mu(B) = \lim_{N\to\infty}\frac{1}{N} \sum_{n=1}^N \ind\{X_n\in B\}
\end{equation}
Throughout, the event $B$ is assumed to be \textit{rare}, entailing that $\gamma$ is very small, typically of order $10^{-4}$ \magenta{or} less (depending on the application at hand). 

\vb

Our interest lies in estimating rare-event probabilities in the context of \textit{models}, so in principle we can do more than \magenta{applying} statistical methods of extreme value analysis to model data; cf.\ \citet{coles2001introduction} for a textbook on Extreme Value Analysis. 
In our setup, the steady-state distribution is not explicitly known; one therefore has to resort to simulation. The na\"ive, Monte Carlo estimator for $\gamma$ is
\[\hat\gamma_\text{MC} := \frac{1}{N} \sum_{n=1}^N \ind\{X_n\in B\},\]
i.e., \textit{the average number of visits to set $B$ until time $N$}, which is known to be extremely inefficient when $B$ is rare; see e.g.\ \citet{asmussen2007stochastic}. Informally, one needs \magenta{prohibitively} many samples in order to obtain a reasonably accurate estimate of $\gamma$; the number of samples required to obtain an estimate of given precision is inversely proportional to $\gamma$. In many cases, especially while working with complex or high-dimensional systems, where the integration of the model is time consuming, such computation might not be feasible.

\vb

An additional complication is that sampling directly from the steady-state distribution can be challenging. In our new method, we settle this issue 
%\kb{we don't really come up with a remedy for sampling from steady-state} 
by dissecting the paths of the underlying Markov chain into \textit{recurrency cycles}. For an arbitrary set $A$, we say that a recurrency occurs each time $(X_n)_{n\in\N}$ crosses $A$ \textit{inwards}, i.e., each time the event $\{X_{n-1}\not\in A,X_n\in A\}$ occurs. Assuming the process is \magenta{in stationarity}, $\gamma$ is equal to \textit{the average amount of time spent in $B$ between \magenta{two} visits to \magenta{the} set $A$, divided by the average length of a recurrency cycle}.

An example of a recurrency cycle is shown in Figure \ref{fig:example_cycle}. It starts at $P_1$ and ends at $P_5$; the time spent in set $B$ is the time spent between states $P_3$ and $P_4$. Note that recurrency is defined with respect to $A$; it is not necessary that the system enters $B$ during a recurrency cycle.

\vb

In our algorithm we separately estimate the numerator (expected time spent in $B$ during a single recurrency cycle) and the denominator (expected length of a single recurrency cycle). Here, two challenges arise. The first concerns the choice of the set $A$. Any $A$ could in principle be used, but in order to maximize the efficiency of the algorithm, it should be chosen so as to minimize the expected time spent between visits to the set $A$. The second challenge is posed by the rarity of visiting $B$ within a cycle. To tackle this issue, we propose the use of  Multilevel Splitting (MLS), see \citet{garvels2000splitting}, \cite{rubino2009rare}, but we remark  that instead of MLS other methods could be chosen. \redd{These alternatives include Genealogical Particle Analysis (see e.g.\ \citet{del2005genealogical}), RESTART (see e.g.\ \cite{villen2011rare}), Adaptive Multilevel Splitting (see e.g.\ \citet{cerou2007adaptive}), fixed-effort and fixed number of successes versions of Multilevel Splitting (see e.g.\ \citet{amrein2011variant}) and Importance Sampling (see e.g.\ \citet{heidelberger1995fast}). We emphasize that we do not seek to compete with any of the aforementioned methods but rather introduce a new overarching framework, in which all these methods can be used to assess stationary performance metrics. We have chosen to work with MLS mostly for its conceptual simplicity and intuitive use.}

\begin{figure}[tb]
 \begin{adjustwidth}{0cm}{}
        \centering
       \includegraphics[width=0.7\linewidth]{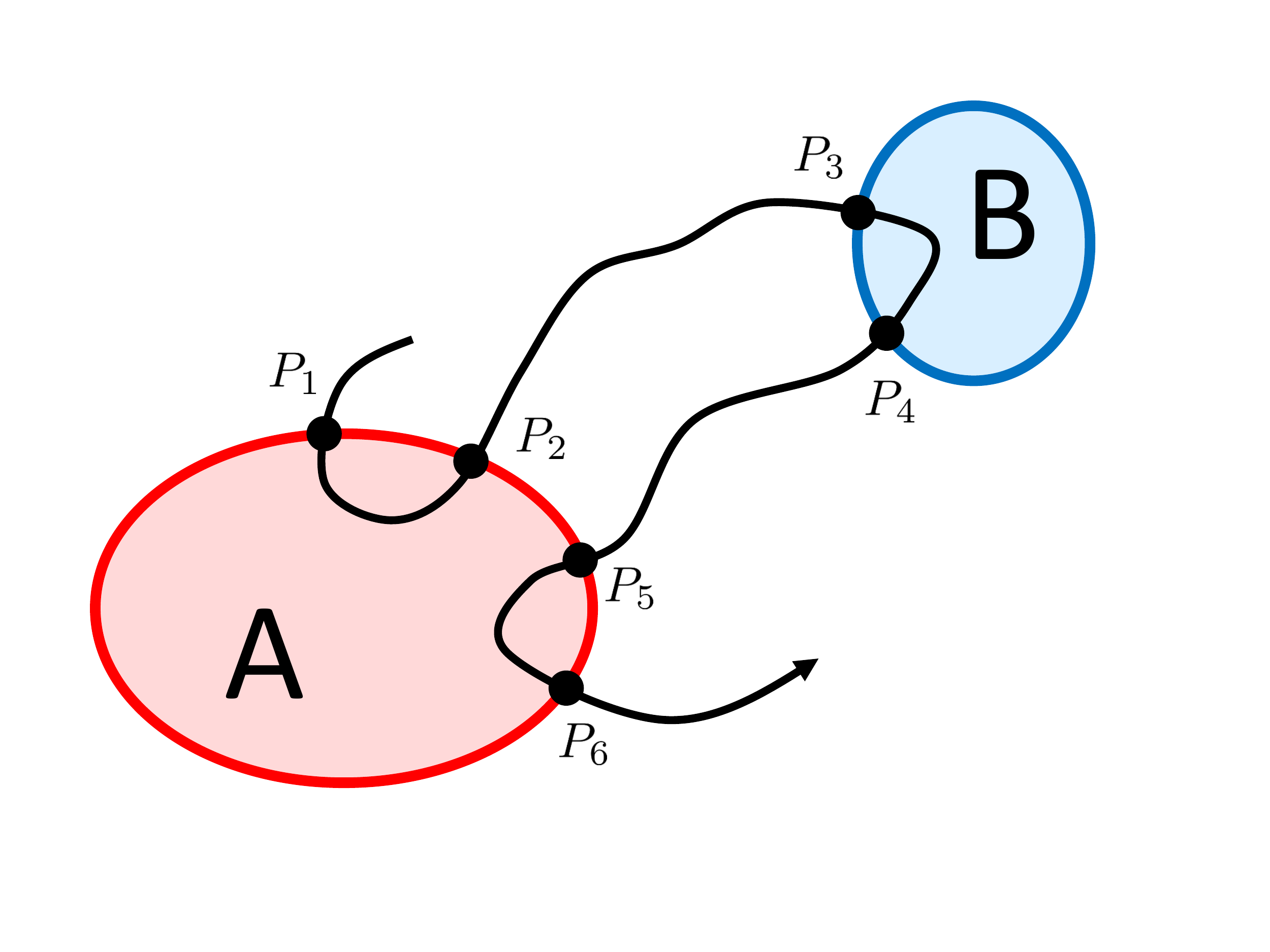}
      \caption{An example of a recurrency cycle. The cycle begins at $P_1$, where the Markov chain enters the set $A$ from the outside, and ends at $P_5$ where the chain enters $A$ again (and the next recurrency cycle begins). 
      }\label{fig:example_cycle}
 \end{adjustwidth}
\end{figure}

\vb

The algorithm we propose is inspired by expressions for steady-state probabilities resulting from the theory of \textit{regenerative processes}. Regeneration instances dissect the path of the process into probabilistically identical, independent segments. For regenerative processes we have that $\gamma$ equals {the average amount of time spent in $B$ in a regeneration cycle divided by the average length of a regeneration cycle}. For more background we refer to \citet{crane1975simulating} and \citet{asmussen2008applied}, or (in a more informal language) \citet{henderson1999can}. In our setup, with its uncountable state space and a steady-state distribution potentially lacking atoms, we cannot straightforwardly construct regeneration points. We therefore develop an approach that relies on the recurrency cycles introduced above, so as to set up a scheme that yields
probabilistically identical (but not necessarily  independent) cycles. We refer to  \citet{goyal1992unified} for an algorithm corresponding to the setting in which the set $A$ consists of finitely many elements (which inspired us to develop our algorithm). We also mention that a large subclass of general (continuous) state-space Markov chains, called \textit{positive Harris}, is regenerative. However, constructing regeneration cycles in this context is typically technically difficult, and in addition the implementation may be computationally inefficient due to excessively long cycle lengths; see \citet{henderson2001regenerative}.

\vb

\red{The manuscript is organized as follows. In Section \ref{s:preliminaries} we discuss preliminaries, such as basic theory of general state-space Markov chains. \dc{We also give} an alternative representation of the parameter $\gamma$ based on the recurrent structure of a Markov Chain in Theorem \ref{thm:reccurence_cycles}. Relying on this alternative representation, in Section \ref{s:algorithm}, we introduce a new algorithm for the estimation of $\gamma$, which we coin {\it Recurrent Multilevel Splitting} (RMS). In Section \ref{s:choice_of_parameters}, we establish (in a simplified setting) the optimal parameters for the RMS algorithm, and provide implementation-related guidelines. \magenta{Theorem \ref{thm:logeff_rms} in Appendix  \ref{appendix:efficiency}
establishes the asymptotic efficiency of the RMS algorithm}. A technical derivation of the optimal parameters is \magenta{given in} Appendix \ref{appendix:optimal_parameters}. In Section \ref{s:experiments} we test the method on a set of numerical examples, we discuss which factors affect the method's performance, and provide heuristics. Finally, in Section \ref{s:discussion} we discuss possible extensions of the algorithm and give a summary. Appendix \ref{appendix:technical_results} consists of \magenta{a collection of required technical results.}}

\iffalse
mainly for simplicity reasons but we argue that the algorithm we present is not splitting-specific and any other rare event simulation method (such as importance sampling) could fit in this framework.
Our idea of the algorithm was inspired by the \textit{Renewal Reward Theorem} for \textit{Regenerative Processes}. Regenerative processes are essentially processes such that their path can be dissected into probabilistically identical, independent cycles. This is somewhat similar to our case, however, our definition of a cycle \textbf{does not} warrant independence. A standard example of regenerative process is a finite state-space Markov chain. There, every visit a certain state of choice is marked as a regeneration. Then parts of the process between entries to that state are iid. An overview of a regeneration method can be found in \citet{crane1975simulating}, \citet{asmussen2008applied}, or with more informal language, in \citet{henderson1999can}.
\fi

%-----------------------------------------------------

\section{Preliminaries}\label{s:preliminaries}

\red{Here we introduce concepts used later in Section \ref{s:algorithm} such as (Harris) recurrence, the stationary measure and \textit{recurrency cycles}.}

\subsection{{Continuous} State-Space Markov Chains}

In this subsection we provide some background on the (well-established) theory of stability of discrete-time Markov chains with a general (continuous) state-space. The underlying theory can be found in textbooks on Markov chains; our notation is in line with the one used in  \citet{meyn2012markov}. 

\vb

The theory of stability for general state-space time-discrete Markov chains differs from the one for its \textit{finite} (or countable) state-space counterpart. Due to the continuous state space, multiple visits to the same state may happen with probability 0. This explains why the classic notion of \textit{irreducibility} and \textit{recurrence} of states has been generalized to {\it sets} (rather than states). In this setting one typically works with the concept of so-called \textit{positive Harris recurrent} chains: sets of states are guaranteed to be visited infinitely often, with in addition a \textit{finite expected return time}. Effectively all Markov chains with an invariant probability distribution are positive Harris (with an exception of pathological, custom-made examples); see \cite[Section~9]{meyn2012markov} for \magenta{a rigorous} treatment of the topic.

\vb
 
Let $(X_n)_{n\in\N}$ be a Markov chain taking values in $\R^d$ with a transition kernel $P(x,\d y)$, meaning that the distribution of $X_{n+1}$ conditional on $X_n=x$ is given by
\begin{equation}\label{eq:transition_kernel}
\p(X_{n+1}\in A \mid X_n = x) = \int_A P(x,\d y)
\end{equation}
for measurable sets $A\subseteq\R^d$. We denote  $P(x,A):= \int_A P(x,\d y)$. Then, the stationary distribution $\mu$ satisfies the relation
\begin{equation}\nonumber
\mu(A) = \int_{\R^d} \mu(\d x)P(x,A).
\end{equation}
For an arbitrary probability measure $\nu$, we define the conditional probability and expectation by $\p_\nu(\cdot) = \p(\cdot\mid X_0\sim\nu)$ and $\Exp_\nu(\cdot) = \Exp(\cdot\mid X_0\sim\nu)$, respectively. In particular, when $\nu$ corresponds to a point mass at $x$, we use the compact notations $\p_x(\cdot) = \p(\cdot\mid X_0=x)$ and $\Exp_x(\cdot) = \Exp(\cdot\mid X_0=x)$, respectively.

\subsection{Recurrent Structure of a Markov Chain}\label{ss:recurrent_structure}

As mentioned in the introduction,  a large class of general state-space Markov chains (more specifically, the class of positive Harris recurrent Markov chains) allows a regenerative structure; see e.g.\ \cite{henderson2001regenerative}. However, for application purposes, it is often difficult to sample the regeneration times. Moreover, even when it is possible to sample these, the implementation is often inefficient due to the long cycle lengths --- in fact, the regeneration may be a rare event itself.

\vb

There are many other ways to decompose a Markov chain into cycles. In this paper we propose to work with cycles that start with an inward crossing of a set $A$ (i.e., entering $A$ from the outside). We denote the time of the $(k+1)$-th inward crossing by $S_{k}$, i.e.,
\begin{equation}\nonumber
S_k := \inf\{ n > S_{k-1} : X_{n-1}\not\in A,X_n\in A\}.
\end{equation}
with $S_{-1}:=0$. Then, we define the paths within the cycles through
\begin{equation}\label{def:cycles}
\mathcal C_k := \big( X_n : S_{k-1}\leq n < S_k-1\big).
\end{equation}
With a $k$-th cycle we associate the \textit{cycle length} and the \textit{cycle origin} (or starting point),
\begin{equation}\label{def:cycle_length}
L_k := S_k-S_{k-1}, \qquad X^A_k := X_{S_{k-1}}.
\end{equation}
\red{We \magenta{call} $A$ the \textit{recurrency set} and $\mathcal C_1,\mathcal C_2, \ldots$ \textit{recurrency cycles}.} Under the assumption that the process $(X_n)_{n\in\N}$ starts in a cycle-stationary regime (that is $X_0\sim \mu$ and $S_0 = 1$.), the pairs $(\mathcal C_1,L_1),$ $(\mathcal C_2,L_2),\ldots$ are identically distributed. However, the cycles \eqref{def:cycles} are generally not independent, as two distinct cycle origins $X^A_k$, $X^A_m$ separated by a short time period $S_{m-1}-S_{k-1}$ tend to be located within the same subregion of the recurrency set. Because of this dependence, the decomposition into recurrency cycles is neither \textit{classic} nor \textit{wide sense regenerative}, see Definition 3.1 and 3.3 in \citet{kalashnikov1994topics}. \redd{The way we define cycles is a special case of the {\it almost regenerative cycles} introduced by \citet{gunther1980almost}. The interested reader is referred to the introduction of \cite{calvin2006semi}, where a more exhaustive account of different regeneration-type methods is outlined.} 

\red{A single recurrency cycle} reflects the behavior of the process in steady-state. To make this claim more precise, define \textit{the total time spent in the set $B$ within the $k$-th cycle}:
\begin{equation}\label{def:reward}
R_k := \sum_{n=S_{k-1}}^{S_{k}-1} \ind\{X_n\in B\}.
\end{equation}
Since (in a cycle-stationary regime) the cycles in \eqref{def:cycles} are identically distributed, so are $R_1,R_2,\ldots$. The following theorem states that the total fraction of time that the process $(X_n)$ spends in the set $B$ is proportional to the expected time spent in $B$ between two consecutive inward crossings into $A$. Define the \textit{frequency of recurrence} $\alpha_A := \p_\mu(X_0\not\in A,X_1\in A)$.

\begin{theorem}\label{thm:reccurence_cycles}
Let $(X_n)_{n\in\N}$ be a positive Harris recurrent Markov chain and let $\mu$ denote its unique stationary probability measure. Let $A$, $B$ be measurable sets such that $\mu(A)\in(0,1)$. Let $L_1$ be as defined in \eqref{def:cycle_length}, $R_1$ as defined  in \eqref{def:reward}, and $T_B := \Exp_\mu R_1$. Then $\Exp_\mu L_1<\infty$,
\begin{equation}\label{def:rec_prop}
\mu(B) = \a_A \cdot T_B
\end{equation}
and $\a_A = (\Exp_\mu L_1)^{-1}$.
\end{theorem}

\begin{proof}
See Appendix \ref{appendix:technical_results}.
\end{proof}

The factorization (\ref{def:rec_prop}) of $\gamma$ from Theorem \ref{thm:reccurence_cycles} is the starting point from which we develop our steady-state rare-event simulation algorithm in Section \ref{s:algorithm}.

\vb

We note that an analogue of Theorem \ref{thm:reccurence_cycles} holds for regenerative processes. Dissection of a Markov chain into \textit{regeneration cycles} has one clear advantage over dissection into recurrency cycles, namely, the regeneration cycles are \emph{independent}. Using this independence, one can easily infer the variance of an estimator based on regeneration cycles. \dc{Nonetheless, \magenta{it is more attractive to use recurrency cycles than regeneration}, as the latter is harder to implement and has a (much) longer expected cycle length. Moreover,
in situations where it is possible to sample from the stationary distribution $\mu$, one can simulate independent paths until the first recurrency cycle has ended, such that the resulting cycles will be independent as well.}

%---------------------------------------------------------------------------------------------------

\section{Recurrent Splitting Algorithm}\label{s:algorithm}

Our algorithm essentially relies on the result from Theorem \ref{thm:reccurence_cycles}, namely the representation of $\gamma$ as a product of two quantities. Thus, we divide our algorithm into two stages: first there is the estimation of $\a_A$ (the frequency of recurrence, \dc{equal to the reciprocal of the expected cycle length}), and secondly the estimation of $T_B$ (the expected time spent in set $B$ within a recurrency cycle). 
\iffalse Development of these algorithms relies on the assumption that one can simulate the future of a Markov chain given any starting point $X_0 = x_0\in\R^d$.\fi

\subsection{Estimation of $\a_A$}\label{ss:estimation:aa}

While it is relatively straightforward to estimate $\a_A$ (for example with a crude Monte Carlo method), the choice of the recurrency set $A$ is non-trivial. In this section we assume that $A$ has already been chosen; the choice of $A$ is discussed in Section \ref{ss:choice_AIF}.

\vb

In typical situations \dc{one can generate  sample paths of $X_n$ by simulation but} it is not possible to \textit{exactly} sample from the stationary distribution. Even though 
the law of $X_n$ converges to $\mu$ weakly, as $n\to\infty$, at any fixed time $n$, the law of $X_n$ is not exactly $\mu$. Perhaps the most straightforward method to estimate $\alpha_A$ in this setting is the method of batch-means. It relies on dissecting  a path of the Markov chain of length $N$ into $m\in\N$ batches \dc{of equal length}, and calculating the sample frequency of entering the set $A$ for each batch. More specifically, with $M := [N/m]$,
\[\hat\a_k := \frac{1}{M} \sum_{n=(k-1)M+1}^{kM} \ind\{X_{n-1}\not\in A,X_n\in A\},\]
and then the batch-means estimator is
\begin{equation}\label{eq:alpha_BM}
\hat\a^\text{BM}_A := \frac{1}{m}\sum_{k=1}^{m}\hat\a_k.
\end{equation}
Let $s_\text{BM}^2$ be the sample variance of $\hat\a_1,\ldots,\hat\a_m$ and $t_{m-1}$ a Student's t distribution with $m-1$ degrees of freedom. Then, due to the `near independence' between the batches, under appropriate regularity assumptions, 
\begin{equation}\label{eq:alpha_BM_conv}
\sqrt{m}(\hat\a^\text{BM}_A - \a)/s_\text{BM} \dd t_{m-1},
\end{equation}
as $N\to\infty$, with  `$\dd$' denoting convergence in distribution.
For more details and background, we refer to e.g.\ \citet{asmussen2007stochastic}.

We remark that when an exact sampling procedure from $\mu$ is available, then it might be more efficient to use the following Monte Carlo estimator. Generate $M$ independent pairs $$(X^{(1)}_0,X^{(1)}_1),\ldots,(X^{(M)}_0,X^{(M)}_1)$$ with (for all $i=1,.., M$)  $X^{(i)}_0\sim\mu$ and $X^{(i)}_1$ distributed according to the dynamics of the Markov chain \eqref{eq:transition_kernel} conditional on the value of $X^{(i)}_0$. The Monte Carlo estimator
\begin{equation}\label{eq:alpha_MC}
\hat\a^\text{MC}_A := \frac{1}{M}\sum_{i=1}^{M}\ind\{X^{(i)}_0\not\in A,X^{(i)}_1\in A\}
\end{equation}
is unbiased, $\Var\,\hat\a^\text{MC}_A = \a_A(1-\a_A)/M$, and, as $M\to\infty$,
\begin{equation}\label{eq:alpha_MC_conv}
\sqrt{M}(\hat\a^\text{MC}_A - \a)/s_\text{MC} \dd N(0,1),
\end{equation}
with $s_\text{MC}^2$ the sample variance.

\vb

Whether exact simulation from $\mu$ is available or not, both methods allow for the construction of confidence intervals based on the weak convergence results  \eqref{eq:alpha_BM_conv} and \eqref{eq:alpha_MC_conv}. It should be clear that the set $A$ should be chosen such that $\alpha_A$ is not prohibitively small, so that the methods   \eqref{eq:alpha_BM} and \eqref{eq:alpha_MC} are computationally efficient. Otherwise, the estimation of $\alpha_A$ would be a rare event simulation problem itself (which we obviously want to avoid).

\subsection{Estimation of $T_B$}\label{ss:estimation_Tb}

The second stage of the algorithm concerns the estimation of $T_B$, as defined in Theorem \ref{thm:reccurence_cycles}. This step is the more challenging one, as the quantity $T_B$ is expected to be very small. We resort to rare-event simulation methods. For clarity of exposition, throughout this section we assume that the chain $(X_n)_{n\in\N}$ is stationary, $S_0 = 0$ and we drop the subscript in $\p_\mu$ and $\Exp_\mu$ (i.e., we write simply $\p$ and $\Exp$, respectively). \dc{We also assume that we can sample from the distribution of the cycle starting point $X_1^A$ (note that $X_1^A, X_2^A, ...$ are all identically distributed).  If we can not, then we \magenta{sample} from $X_1^A$ approximately; this is discussed in Section \ref{ss:estimation_gamma}}. We first introduce some notation; we define $p_B := \p(\tau_B<\tau^\text{\rm in}_A)$, with
\begin{equation}\nonumber
\tau_B := \inf\{n>0: X_n\in B\}, \quad \tau^\text{\rm in}_A := S_1 = \inf\{n>0: X_{n-1}\not\in A,X_n\in A\},
\end{equation}
and
\begin{equation}\label{def:Rp}
R_+ :\eqd \big(R_1\mid R_1>0)
\end{equation}
\red{with `$\eqd$' denoting equality in distribution.} Note that $\tau^\text{\rm in}_A-1$ marks the end of the \dc{first} recurrency cycle. Since $\{R_1>0\} = \{\tau_B<\tau^\text{in}_A\}$, $p_B$ is the probability of reaching $B$ within a cycle, and $R_+$ is a random variable distributed as the total time spent in the set $B$ within a cycle conditioned on the cycle reaching set $B$. As was noted in \citet{garvels2000splitting},
\begin{align*}
\Exp(R_1) = \p(R_1>0)\cdot\Exp(R_1 \mid R_1>0).
\end{align*}
This entails that
\begin{equation}\label{eq:Tb_decomposition}
T_B = \p(\tau_B<\tau^\text{\rm in}_A) \cdot \Exp(R_1\mid \tau_B<\tau^\text{\rm in}_A) = p_B \cdot \Exp R_+
\end{equation}
The estimation of $p_B$ is a classic rare-event simulation problem, for which various methods have been developed. Following \cite{garvels2000splitting}, we propose to use a Multilevel Splitting (MLS) algorithm to estimate $T_B$ (but, as we mentioned before, other approaches could be followed as well). There are a number of variations of the MLS algorithm; we chose to rely on its simplest version (called `Fixed Splitting'). The following exposition aligns with \citet{amrein2011variant}.

\iffalse
For clarity of the exposition, we slightly simplify the setting of our problem. Let $(X_n)_{n\in\N}$ be a Markov chain taking values in $\R^d$, and $A, B$ be measurable sets. Let $\tau_C := \inf\{n>0: X_n\in A\}$ be the time of the first visit to set $A$ and $\tau_B$ defined analogously. We are interested in estimation of the probability of visiting $B$ before reaching $A$ (starting from $X_0\sim\nu$, with $\nu(A) = 1$)
\begin{equation}\label{eq:p_MLS}
p := \p_\nu(\tau_B < \tau_A).
\end{equation}
This setting differs from the exposition in Section \ref{ss:recurrent_structure}, however it is possible to rephrase our problem such that it fits in this framework. This is analogous to the proof of Theorem \ref{thm:reccurence_cycles} by defining a new Markov chain, which keeps one-step past, i.e. $Z_n = (X_{n-1},X_n)$ and redefining sets $A$ and $B$ accordingly.
As already mentioned, when $p_B$ is small, it is inefficient to estimate it with Monte Carlo method because a lot of computational effort will be wasted on the simulation of uninteresting paths. 
\fi

\vb

As mentioned, the na\"{\i}ve Monte Carlo method is inefficient for the estimation of small probabilities, because of the computational effort wasted on simulating irrelevant paths. 
The core idea behind the MLS method is to \textit{split} the path of the process when it approaches $B$. This way, we have more control over the simulation, by forcing the process into interesting regions. In order to implement the MLS algorithm, one must first choose an \textit{importance function} $H:\R^d\to[0,1]$ which assigns an \textit{importance value} to every possible state. $H$ should be chosen such that $H(x)=1$ if and only if $x\in B$ and $H(x) = 0$ for $x\in A$. We postpone the discussion about the choice of the importance function to Section \ref{ss:choice_AIF}.

\vb

We now formally introduce the MLS algorithm. First divide the interval $[0,1]$ into $m$ subintervals with endpoints:
\begin{equation}\nonumber
0 = \ell_0 < \ell_1 < \ldots < \ell_m = 1,
\end{equation}
and define the corresponding stopping times and events
\begin{equation}\label{def:taukDk}
\tau_k := \inf\{n\geq0 : H(X_n)\geq \ell_k\},\:\: D_k := \{\tau_k < \tau_A^\text{\rm in}\};
\end{equation}
for $k\in\{0,\ldots, m\}$. Note that $\tau_k$ is the first time an importance value greater or equal to $\ell_k$ has been reached; \red{in particular $\tau_m = \tau_B$ and $\tau_0=0$, so that $X_{\tau_0}\eqd X^A_1$.} Finally let \[p_k := \p(D_k\,|\,D_{k-1}),\:\:\:k\in\{1,\ldots, m\},\] and $p_0 = 1$, \red{to which we refer as \textit{conditional probabilities}}. From the definition \eqref{def:taukDk} we have $\p(D_m) = p_B$ and since $D_0\subseteq D_1\subseteq\ldots\subseteq D_m$, we conclude
\[p_B = \prod_{k=0}^m p_k.\]
Finally, define \textit{splitting factors} \redd{$n_0,n_1,\ldots,n_{m} \in \N$}, representing the number of independent continuations of the process that are sampled when reaching the \magenta{respective} importance levels. \redd{Here $n_0$ plays a special role,} as it is a number of independent MLS estimators; the final estimator will be a mean of $n_0$ independent MLS estimators. By virtue of this independence, we are able to estimate the variance of the final estimator. \redd{For simplicity, in the following it is assumed that $n_0=1$.}

\begin{algorithm}[Multilevel Splitting]\label{alg:MLS} \ 
\begin{enumerate}
\item\label{step:quiet_ass} Set $k := 0$, $r_0 := 1$, sample $X^{1}_0 \sim X^A_1$.
\item\label{step:fs} In \magenta{the} $k$-th stage we have a sample of $r_k$ entrance states $(X^{1}_k,\ldots,X^{r_k}_k)$, where we denote $$X^{i}_k := X^{i}_{\tau^{i}_{k}}.$$ For each state $X^{i}_{k}$ generate $n_k$ independent path continuations until $\min \{\tau_{k+1},\tau^\text{\rm in}_A\}$. The number of paths for which the event $D_{k+1}$ occurred is denoted by $r_{k+1}$. Store all $r_{k+1}$ states $X^{i}_{k+1}$, for which the event $D_{k+1}$ occurred, in memory.
\item If $r_{k+1} = 0$, then stop the algorithm and put $\hat p_B := 0$, $\hat T_B := 0$.
\item If $k < m-1$, then increase $k$ by one and go back to step \ref{step:fs}; otherwise put
\begin{equation}\label{eq:p_hat}
\hat p_B := \frac{r_m}{\prod_{k=0}^{m-1} n_k}.
\end{equation}
\item If $r_m=0$, then return $\hat T_B = 0$; otherwise, for each state $X^{i}_{m}$ generate $n_m$ independent path continuations until $\tau^\text{\rm in}_A$. For each of these $r_m n_m$ continuations \magenta{record} the time spent in set $B$: $$\hat R_+^{(j)} := \sum_{k=\tau_m}^{\tau^\text{\rm in}_A-1}\ind\{X_k\in B\}.$$ Calculate \magenta{the} total time spent in $B$ by
\begin{equation}\nonumber
r_{m+1} := \sum_{j=1}^{r_m n_m}\hat R_+^{(j)}
\end{equation}
\item The final estimator is
\begin{equation}\label{eq:Tb_hat}
\hat T_B := \frac{r_{m+1}}{\prod_{k=0}^m n_k}
\end{equation}
\end{enumerate}
\end{algorithm}

\begin{theorem}\label{thm:unbiasedness}
The estimators $\hat p_B$ and $\hat T_B$, as defined in \eqref{eq:p_hat} and \eqref{eq:Tb_hat}, are unbiased estimators for $p_B$ and $T_B$ respectively.
\end{theorem}

The following proof is based on notes of the {\it Summer School in Monte Carlo Methods for Rare Events} that took place at Brown University, Providence RI, USA in June 2016 (authored by J.\ Blanchet, P.\ Dupuis, and H.\ Hult). It is noted that various alternative derivations can be constructed; see e.g.\ \cite{asmussen2007stochastic}.

\begin{proof}[Proof of Theorem \ref{thm:unbiasedness}]
\iffalse
Unbiasedness of $\hat p_B$ is covered in literature in a slightly different setting, with $\tau_A$ instead of $\tau_A^\text{\rm in}$. The proof however is analogous. See \citet{asmussen2007stochastic}. We give a very intuitive proof here. 
\fi
Let $\bar X_{i,j}$ be labeling  all descendants of the original particle, with $i$ indexing time and $j$ indexing the descendant. All descendants $\bar X_{\cdot,j}$ are identically distributed (but not independent). Now suppose that each particle has an evolving weight $w_{i,j}$. Concretely, this means that when a particle crosses a threshold $\ell_k$, it is split into $n_k$ particles and its weight is divided equally among its descendants (i.e., each of them obtaining a share $1/n_k$ of $w_{i,j}$). Each particle that reaches the set $B$ has been split $m$ times, and its weight is thus $1/\prod_{k=1}^m n_k$. For particles that did not reach set $B$, we artificially split these particles (keeping them in $A$) for the remaining thresholds so that the total number of particles is $\prod_{k=1}^m n_k$, each of equal weight. Then, using the fact that the descendants are identically distributed, we obtain
\begin{align*}
\Exp \hat T_B & = \Exp \Bigg( \sum_{j=1}^{\prod_{k=1}^m n_k} \frac{1}{\prod_{k=1}^m n_k} \sum_i \ind \{\bar X_{i,j} \in B\}\Bigg) = \Exp \sum_i \ind \{\bar X_{i,1} \in B\} = T_B.
\end{align*}
Analogously, $\Exp\hat p_B = p_B$, which ends the proof.
\end{proof}

We remark that, with $r_1,\ldots,r_m$ as defined in Algorithm \ref{alg:MLS},  the same arguments as the ones featuring in the proof of Thm.\ \ref{thm:unbiasedness} imply the unbiasedness of the estimators for $\p(D_k)$:
\begin{equation}\label{eq:unbiased_Pk}
\Exp \bigg(\frac{r_k}{\prod_{i=0}^{k-1}n_i}\bigg) = \p(D_k) = p_1\cdots p_k.
\end{equation}

\subsection{Estimation of $\gamma$}\label{ss:estimation_gamma}

As already mentioned at the beginning of Section \ref{s:algorithm}, the final estimator for $\gamma$ \dc{is the} product $\hat\gamma := \hat\a_A\cdot\hat T_B$. In the description the MLS algorithm, in Step \ref{step:quiet_ass}, we tacitly assumed that we \magenta{can sample the recurrency cycle origin $X^A_1$}. \magenta{As this is typically not the case, we sample $X_1^A$ approximately, in the following way.} During the estimation of $\a_A$ with the batch-means method \eqref{eq:alpha_BM} we store each inwards crossing to the set $A$ and we bootstrap these states in Step \ref{step:quiet_ass} of Algorithm \ref{alg:MLS}. We thus end up with the following algorithm for estimating the rare-event probability $\gamma$, as defined in \eqref{def:gamma}.

\begin{algorithm}[Recurrent Multilevel Splitting]\label{alg:RMS}\
\begin{enumerate}
\item Choose a recurrency set $A$ satisfying the assumptions of Theorem $\ref{thm:reccurence_cycles}$ and an importance function $H:\R^d\to[0,1]$.
\item\label{step:aA} Estimate $\a_A$ using the batch-means method $\eqref{eq:alpha_BM}$, and return $\hat\a_A$. Store the locations of the cycle origins in the set $\mathcal S_\text{\rm rec} := \{X^A_1,X^A_2,\ldots\}$.
\item\label{step:Tb} Estimate $T_B$ using the Multilevel Splitting algorithm (Algorithm $\ref{alg:MLS}$); in Step $\ref{step:quiet_ass}$ sample the origin $X_0^1$ uniformly from $S_{\rm rec}$. The output is $\hat T_B$.
\item The final estimator is
\begin{equation}\label{def:rec_estimator}
\hat \gamma := \hat\a_A \cdot \hat T_B
\end{equation}
\end{enumerate}
\end{algorithm}

It is assumed that the set $S_\text{rec}$ is `representative enough' to make sure  that resampling from $S_\text{rec}$ can be interpreted as taking i.i.d.\ \dc{samples} of $X^A_1$ in the stationary regime. \red{Under this assumption, the estimators $\hat\a_A$, $\hat T_B$ are independent and the variance of $\hat\gamma$ can be inferred using the sample variance of $\hat\a_A$ and $\hat T_B$. However, in our numerical experiments in Section \ref{s:experiments} we do not assume this independence \dc{to get an estimate of the variance. Instead we run Algorithm \ref{alg:RMS} multiple times, resulting in multiple estimates $\hat \gamma$ from which we} obtain a reliable estimate for the variance of $\hat\gamma$. \magenta{For implementation details, see Section \ref{ss:implementation_details}.}}

\section{Choice of Parameters}\label{s:choice_of_parameters}

In a rare-event setting, both the expectation and the variance of an estimator are very small, so that the variance itself is not a meaningful measure of accuracy. Instead, it makes sense to look at its value relative to the expectation, i.e., the {\it Relative Error} (RE):
\[\RE^2(\hat\gamma) := \Exp(\hat\gamma-\gamma)^2/\gamma^2.\] An estimator with a lower relative error is not necessarily preferred; a more meaningful criterion involves the corresponding total computational time (or: {\it workload}), which we denote $W(\hat\gamma)$; \red{see the beginning of Section \ref{ss:implementation_details} for more details. In the following section we consider a setting, in which we can derive \dc{optimal parameters of the MLS estimator by minimizing the workload under a constraint on the relative error (i.e., $\RE^2(\hat\gamma) \leq \rho$ for a given accuracy $\rho>0$).}}

\subsection{\red{Simplified Setting}}\label{ss:simplified_setting}

\red{\redd{Due to possible dependencies between the number of successes $r_1,\ldots,r_m$,} there is no tractable general expression for the variance of MLS estimator.} A typical assumption made in the literature is to assume some sort of independence between them, and to study the variance afterwards. With $\tau_k,D_k$ defined as in \eqref{def:taukDk} and $R_+$ as defined in \eqref{def:Rp}, \magenta{we assume}
\begin{enumerate}[label=\assenum*,ref=\assenumref*]
\item\label{ass:amrein} for all $k\in\{1,\ldots,m\}$,
\begin{equation}\nonumber
\p(D_k \mid D_{k-1},X_{\tau_{k-1}}) \equiv \p(D_k\mid D_{k-1}) = p_k
\end{equation}
\item\label{ass:trare} for all $X_{\tau_m}$,
\begin{equation}\nonumber
\big(R_1 \mid R_1>0, X_{\tau_m}\big) \eqd (R_1 \mid R_1>0) =: R_+
\end{equation}
\end{enumerate}
Assumption \eqref{ass:amrein} has been proposed in \cite{amrein2011variant}. It states that the probability of reaching the $k$-th importance level, given the $(k-1)$-st level has been reached, is constant over all possible entrance states. Assumption \eqref{ass:trare} states that the time spent in \magenta{the} rare set $B$ within a cycle, conditioned \magenta{on} the set $B$ has been reached, does not depend on the position of the entrance state \dc{to $B$. In principle, we have the possibility} \red{to choose \magenta{the} set $A$ and \magenta{the} importance function $H(\cdot)$ such that Assumption \eqref{ass:amrein} is satisfied; see the discussion in Section \ref{ss:choice_AIF}. Whether Assumption \eqref{ass:trare} holds or not is \magenta{effectively} problem specific, \magenta{in the sense that we do not have control over it due to the fact that the set $B$ is given}. We argue that for a large class of problems there exists a most likely point of entry $X_{\tau_B}$ to $B$, which implies  \eqref{ass:trare} approximately. We emphasize that \magenta{Assumptions} \ass\ are not required \dc{for the RMS algorithm to work,} but if they are fulfilled, optimality results can be derived.} Under \ass\ we find the squared relative error of $\hat T_B$:
\begin{equation}\label{eq:SRE_Tb}
\RE^2(\hat T_B) = \sum_{k=1}^{m} \frac{(1-p_k)/p_k}{\prod_{j=0}^{k-1}n_jp_j} + \frac{\RE^2(R_+)}{\prod_{j=0}^{m}n_jp_j}.
\end{equation}
We derive \eqref{eq:SRE_Tb} in Appendix \ref{appendix:technical_results}. \red{Following the approach of \citet{amrein2011variant}, in Appendix \ref{appendix:optimal_parameters} we derive the optimal parameters $m,p_1,\ldots,p_m,n_0,\ldots,n_m$ for the MLS algorithm; \magenta{here, optimality refers to the property} that the expected computational time is minimized under the constraint for the relative error \magenta{$\RE^2(\hat T_B)\leq \rho$} for a given accuracy $\rho>0$.} It is worth noting that the optimal number of thresholds $m$ is roughly equal to $|\log p_B|$ with conditional probabilities $p_k$ all equal to approximately $0.2$. What is more, the optimal solution satisfies $n_k p_{k+1} = 1$ for $k\in\{1,\ldots,m-1\}$, so we can choose $n_k = 5$. This so-called \textit{balanced growth} (see \cite{garvels2000splitting}) ensures that, on average, $n_0$ paths are sampled in each stage of the algorithm (with an exception of the last stage, which corresponds to the estimation of $R_+$). The optimal workload reads
\begin{equation}\label{eq:optimal_workload}
W(\hat T_B) = \frac{1}{q}\bigg(\frac{c\,|\log p_B|}{\sqrt{2c-1}} + \RE(R_+)\bigg)^2
\end{equation}
with a constant $c$ defined as below display \eqref{eq:optimal_parameters}. As already mentioned, \magenta{a rigorous} derivation of this result can be found in Appendix \ref{appendix:optimal_parameters}, and \magenta{the} exact values of \magenta{the} optimal parameters $m,p_1,\ldots,p_m,n_0,\ldots,n_m$ in Eq.\ \eqref{eq:optimal_parameters}. \red{In all our numerical experiments in Section \ref{s:experiments}, we spend \magenta{an initial portion of computational time} on a rough estimation of $p_B$ and $\RE(R_+)$ in order to find \magenta{a sufficiently accurate approximation of} the optimal parameters. See Section \ref{ss:implementation_details} for \magenta{a more detailed account of the implementation details.}}

The optimal workload in \eqref{eq:optimal_workload} is proportional to $(\log p_B)^2$, which  offers  a huge gain in efficiency, compared with the Monte Carlo method \eqref{eq:MC_workload} (whose workload is inversely proportional to $p_B$). We derive efficiency results in Appendix \ref{appendix:efficiency}; \magenta{in particular, Theorem \ref{thm:logeff_rms} proves that RMS is logarithmically efficient under specific assumptions}.

%----------------------------------------------------------------

\subsection{Choice of Recurrency Set and Importance Function}\label{ss:choice_AIF}

In Section \ref{ss:simplified_setting} we have seen that under Assumptions \ass, the MLS method is particularly efficient. \red{As already mentioned, the \magenta{level up to which Assumption \eqref{ass:amrein} is fulfilled} depends on both the choice of the recurrency set and the importance function; we thus aim to choose $A$ and $H(\cdot)$ in such a way that \eqref{ass:amrein} is approximately satisfied.} At the same time, we would like to choose $A$ so as to maximize $\a_A$, so that the batch-means estimator $\hat \a_A$ (as defined in \eqref{eq:alpha_BM}) is computationally efficient as well. These two requirements are often conflicting and one must in the end strike a proper balance between them.

For each $k$, Assumption \eqref{ass:amrein} concerns the choices of both $A$ and $H(\cdot)$. \dc{However, it} implies a property that relates to the choice of $A$ only, namely, the probability of reaching set $B$ within a recurrency cycle is independent of the initial point:\[\p(\tau_B<\tau^\text{in}_A \mid X^A_1) \equiv p_B.\] 
\dc{Thus, Assumption \eqref{ass:amrein} implies that}
\begin{equation}\label{ass_A}
X^A_1 \overset{\rm d}{=} \big(X^A_1\mid R_1>0\big) =: X^A_+;
\end{equation}
informally, there is independence between the origin of the cycle on one hand, and the random variable $\ind\{R_1>0\}$ (indicating whether set $B$ has been reached within a cycle) on the other hand. \red{Intuitively, the smaller the set $A$ is, the more closely \eqref{ass_A} is satisfied but also, the smaller $\a_A$ is. In particular, \eqref{ass_A} trivially holds when $A$ consists of one point only, but then $\a_A = 0$.} In Section \ref{sss:OU2} we give an example of a setting in which \eqref{ass_A} is violated, but one can imagine that in many situations \eqref{ass_A} `roughly holds'.  \dc{Thus, for practical purposes, it is desirable that \emph{the set $A$ maximizes $\a_A$ while it also approximately satisfies  \eqref{ass_A}}. In full generality, it is not an easy task to \magenta{fulfill} both aims.}

A poorly chosen importance function will lead the split particles into uninteresting regions, or it will force the paths to hit the rare set in an unlikely fashion. This potentially leads to low efficiency of the MLS algorithm. Given that we have already chosen a set $A$ satisfying \eqref{ass_A}, there exists an importance function guaranteeing  \eqref{ass:amrein} to be satisfied:
\begin{align*}
H(x) := \p_x(\tau_B\leq\tau^\text{in}_A),
\end{align*}
Of course this insight is of theoretical value only: if we knew the quantity on the right hand side, then we would not even have to use the MLS algorithm. However, also \[H_g(x) := g(\p_x(\tau_B\leq\tau^\text{in}_A)),\] with $g:[0,1]\to\R$ any increasing function, satisfies \eqref{ass:amrein}. This already gives a helpful guideline for the choice of $H$. Namely, \emph{the states from which it is more likely to visit $B$ before returning to $A$ should have larger importance}. When an approximation or asymptotic behavior of $\p_x(\tau_B\leq\tau_A^\text{in})$ is available it might be useful to use it as an importance function. In \citet{dean2009splitting} a large-deviations based approach to the choice of importance function is discussed.

Sometimes, a so-called \textit{distance-based} importance function can be a good choice. This function is basically \[H(x) :=\text{\rm dist}(x,B) = \inf\{\|x-a\|:a\in B\},\] normalized in such a way that $H(x)=1$ iff $x\in B$ and $H(x)=0$ for $x\in A$. This importance function can be a good choice for systems whose paths conditioned on $\{\tau_B<\tau^\text{in}_A\}$ are effectively gradually driven towards $B$. \redd{In contrast}, distance-based importance function will be a poor choice for systems for which it is most likely to reach rare set $B$ by first getting away from it. In Section \ref{s:experiments} we include examples of problems for which a distance-based importance function is a good choice, but also one in which it does not work \magenta{well}. 

In some cases we may have already chosen \dc{a particular} \textit{shape} of the set $A$ (e.g.\ an ellipsoid, half-space, or multidimensional cube) which can be parametrized by a single parameter $\ell\in\R$. Even better, if we have already chosen an importance function, then a level set \[A(\ell) = \{x\in\R^d : H(x)\leq \ell\}\] could be a good choice. In any case, we should choose $\ell$ to maximize $\alpha_{A(\ell)}$. We propose to use a crude estimator to find $\ell^*$: we find a maximizer of $\alpha_{A(\ell)}$ by putting
\begin{equation}\label{eq:lstar}
\hat\ell^* := \argmax\bigg\{\sum_{n=0}^{N} \ind\{X_n \not\in A(\ell), X_{n+1} \in A(\ell)\}\bigg\}.
\end{equation} 

\textit{Quantile validation.} While it is not clear in general how to choose $A$ such that it satisfies \eqref{ass_A}, one can statistically test whether \eqref{ass_A} holds after the choice of $A$ has been made. We now propose one particular method to do so that can be used in combination with the RMS algorithm. In Step \ref{step:aA} of Algorithm \ref{alg:RMS} calculate and store the \textit{maximum importance attained within cycles}, i.e., \[H_k^\text{\rm max} := \max\{H(x) : x\in \mathcal C_k\},\] with $\mathcal C_k$ as defined in \eqref{def:cycles}. Assuming a good importance function has been chosen, the cycle origins corresponding to the highest importance should also be approximately distributed as $X^A_+$. This gives us means of comparing the distributions of $X^A_1$ and $X^A_+$. Let $N_\text{rec}$ be the total number of pairs $(X^A_k,H^\text{max}_k)$ obtained in Step \ref{step:aA} of Algorithm~\ref{alg:RMS}. Let \[\sigma:\{1,\ldots,N_\text{rec}\}\to\{1,\ldots,N_\text{rec}\}\] be a permutation ordering $(H^\text{max}_k)_{1\geq k\geq N_\text{rec}}$ into a non-decreasing sequence, i.e.,
\[H^\text{max}_{\sigma(1)} \leq H^\text{max}_{\sigma(2)} \leq \ldots \leq H^\text{max}_{\sigma(N_\text{rec})}\]
Now choose a \magenta{$q\in(0,1)$} and let
\begin{equation}\label{def:Sq}
S^q_\text{rec} := \big\{ X^A_{\sigma([(1-q)N_\text{rec}])},\ldots,X^A_{\sigma(N_\text{rec})}\big\}
\end{equation} 
That is, $S^q_\text{rec}$ is a subset of $S_\text{rec}$ which contains the cycle origins corresponding to \magenta{the fraction $q$ of values with highest importance}. In particular $S^{1}_\text{rec} = S_\text{rec}$. Then $S_\text{rec}$ and $S_\text{rec}^q$ (for small $q$) can be thought of as sets of samples from the random variables $X^A_1$ and $X^A_+$, respectively. Various tests can now be performed, to  compare e.g.\ the means or variances; alternatively   QQ-plots can be made, or histograms can be compared.

\section{Numerical Experiments}\label{s:experiments}

The aim of this section is to test the RMS method on a series of \magenta{specific} examples. The examples range from simple cases, where the ground truth is known, to more complicated dynamical systems, \dc{where the ground truth is unknown and we can only compare to estimates obtained} with Monte Carlo (MC) methods. In Section \ref{sss:OU2} we also carefully look into an example where the RMS method (with a \magenta{na\"{\i}ve} choice of \magenta{the} importance function) does not perform \magenta{that well}; we discuss why this was \magenta{to be} expected. It will be seen throughout that RMS is superior to MC in terms \magenta{of the} computational time needed to achieve a \dc{desired} level of accuracy; in extreme cases, like in Section \ref{ss:franzke}, the RMS method can be three orders of magnitude faster than MC (and the \dc{efficiency gain} is expected to be even greater as $\gamma$ decreases).

\subsection{Implementation Details}\label{ss:implementation_details}

As already mentioned in Section \ref{s:choice_of_parameters}, the relative error of an estimator is not always a meaningful measure of \magenta{its} performance, as it does not take the workload into account. We \magenta{therefore} compare RMS with MC using the ratio of \textit{work normalized squared relative errors}; see e.g. \cite{kroese2013handbook}. In particular, we define
\begin{equation}\label{def:eff}
\Eff(\hat\gamma) = \frac{W(\hat\gamma^{\rm MC})}{W(\hat\gamma)} \cdot \frac{\RE^2(\hat\gamma^{\rm MC})}{\RE^2(\hat\gamma)}.
\end{equation}
\iffalse
Let $\hat\gamma^{(1)},\ldots,\hat\gamma^{(N)}$ be $N$ independent copies of an estimator $\hat\gamma$. Let $\hat\gamma_N = \tfrac{1}{N}\sum_{i=1}^N\hat\gamma^{(i)}$; then
\[\RTV(\hat\gamma_N) = \RTV(\hat\gamma).\]
Since increasing the number of independent copies of the estimator does not change its relative time variance product, then we can define the \textit{efficiency gain},
\fi
This value can be interpreted as \dc{the ratio of the computational cost of MC to the cost of RMS when both methods reach the same accuracy (same relative error)}. Clearly, the larger $\Eff(\hat\gamma)$ is, the more efficient the RMS method is in comparison with Monte Carlo. 

\vb

In each of our experiments, the underlying Markov chain $(X_n)_{n\in\N}$ represents the numerical solution to a $d$-dimensional Stochastic Differential Equation (SDE) using an explicit Euler scheme, with \redd{time step} $h>0$; see e.g.\ \cite{kloeden1992numerical}. \redd{We remark that the time discretization potentially has a significant effect on a the underlying value of $\gamma$, especially in the rare-event setting; see the recent systematic study \cite{bisewski2018simulation}. However, in the context of this article we only focus on discrete recursions  that arise from numerical time integration schemes. For these recursions we compare RMS with the corresponding Monte Carlo results; we do not aim at studying the behavior as $h\downarrow 0$.}

\vb

\redd{Notice that our method relies on properties of discrete-time processes, in particular in the definition of the recurrency cycles. More specifically, in the corresponding continuous-time model recurrency cycles are ill-defined,  as a set may be entered and left infinitely often in a time interval of finite length. This feature could potentially lead to computational issues when working with a small time step $h$.  However, one can easily circumvent the problem and still integrate the process with arbitrarily small $h_0$ but store values every $h>h_0$. Note that the discretization error depends only on $h_0$ (and not $h$), since $h_0$ determines the stationary distribution. In fact, this is what we do in Section \ref{ss:franzke}, where the process is integrated with $h_0 = 10^{-4}$ but it is stored only every $h=10^{-2}$.}

In each experiment the rare event $B$ is a half-space parametrized by $u\in\R$:
\begin{equation}\nonumber
B(u) = \{(x_1,\ldots,x_d)\in\R^d : x_1\geq u\}.
\end{equation}
In other words, the probability under consideration \magenta{corresponds to the the first dimension
attaining high values in stationarity:}
\begin{equation}\label{def:gamma_u}
\gamma(u) := \p_\mu(X_0\in B(u))
\end{equation}
for large $u$. \dc{Furthermore, in each experiment we choose the recurrency set $A$ to be} a half-space parametrized by $\ell$ (\magenta{where} the value of $\ell$ is chosen depending on the particular experiment):
\begin{equation}\label{def:A_u}
A(\ell) = \{(x_1,\ldots,x_d)\in\R^d : x_1\leq \ell\}.
\end{equation}
We use a distance-based importance function, i.e.,
\begin{equation}\label{def:IF_u}
H(x_1,\ldots,x_d) = \begin{cases}
0, & x_1\leq 0 \\
x_1/u, & x_1\in(0,u) \\
1, & x_1\geq u
\end{cases}
\end{equation}
%\newcommand{\numset}{{\rm(\ref{def:B_u}-\ref{def:IF_u})}}%
%Each numerical experiment is carried out in the following way.
%First, we estimate $\gamma(u)$ using Monte Carlo with a time-lag determined in the following fashion: first we roughly estimate the temporal autocorrelation from a time series of the process in the stationary regime, i.e.
%\[c(k) := \Cov_\mu(X_0,X_n)/\Var_\mu(X_0)\]
%for $k=\{1,2,\ldots\}$. The time-lag for the Monte Carlo is then defined as
%\[n_\text{lag} := 10 \cdot (\max\{k>0 : |c(k)|\geq 0.2\}+1)\]
%So $n_\text{lag}$ is 10 times bigger than the time needed for the autocorrelation to drop below 0.2.
We now provide more details on our implementation of Algorithm \ref{alg:RMS}. In Step \ref{step:aA}, we estimate $\a_A$ using the method of batch means as in \eqref{eq:alpha_BM}; the number of iterations of the Markov chain $N$ is chosen such that $S_{\rm rec}$ consists of roughly $10^4$ inwards crossings of $A$. In Step \ref{step:Tb}, we want to choose parameters $m,n_0,\ldots,n_m,\ell_1,\ldots,\ell_m$ for the Multilevel Splitting in such a way that the workload is minimized and the resulting estimator satisfies
\begin{equation}\label{eq:desired_RE}
\RE(\hat T_B)=5\cdot10^{-3}.
\end{equation}
We run a pilot MLS with many intermediate thresholds ($m=20$). The pilot gives us rough estimates of $p_B$, $T_B$ and $\RE(R_+)$. We put the number of thresholds $m$ and splitting factors $n_0,\ldots,n_m$ as in \eqref{eq:optimal_parameters}; we emphasize that the optimal $n_0$ is also determined by the desired squared relative error $\rho$. We find the intermediate thresholds $\ell_1,\ldots,\ell_m$ following the log-linear interpolation approach from \citet{wadman2014separated}. Assuming \ass\ are satisfied, the MLS method with these parameters should give the desired relative error, as in \eqref{eq:desired_RE}. We note that in the pilot we use the variant of MLS called `Fixed Number of Successes' developed by \citet{amrein2011variant}.

\vb

The final estimator $\hat\gamma$ is the mean of $N=100$ independent replicas \dc{$\hat\gamma^{(1)}, \ldots,\hat\gamma^{(N)}$} of the RMS estimator \dc{(\ref{def:rec_estimator})} with parameters as discussed above; i.e.
\begin{equation}\nonumber
\hat\gamma := \frac{1}{N}\sum_{i=1}^N \hat\gamma^{(i)}
\end{equation}
This additional `Monte Carlo wrapper' around the RMS method enables us to approximate the relative error $\RE(\hat\gamma)$ with
\begin{equation}\label{def:sampleRE}
\RE^2(\hat\gamma) \approx \frac{1}{N-1}\sum_{i=1}^N \bigg(\frac{\hat\gamma^{(i)}}{\hat\gamma} -1\bigg)^2,
\end{equation}
and \dc{we can approximate $\RE(\hat\a_A)$ and $\RE(\hat T_B)$ in a similar way}. For each experiment we present a table with results corresponding to  multiple values of the threshold $u$. Each table displays the final estimator $\hat\gamma$ as well as its estimate for $\RE(\hat\gamma)$, as in \eqref{def:sampleRE}, and $\Eff(\hat\gamma)$, as in \eqref{def:eff} based on the run of an MC estimator $\hat\gamma^{\rm MC}$.

\vb

Various checks can be done in order to assess the reliability of the estimator $\hat\gamma$. In each table we additionally give the estimate for $\RE(\hat T_B)$; if it matches the desired relative error, i.e. $\RE(\hat T_B)\approx 5\cdot10^{-3}$, \magenta{then} \dc{this is} an indication that Assumptions \ass\ are satisfied. When $\RE(\hat T_B)$ is larger than desired, it might be a result of poorly chosen intermediate thresholds $\ell_1,\ldots,\ell_m$; we propose to verify, after the algorithm has been executed, whether the estimates for all the intermediate probabilities $p_1,\ldots,p_m$ roughly equal the optimal $p_\text{opt}\approx 0.20$. If this is the case and we still get a particularly large $\RE(\hat T_B)$, this is an indication that either the recurrency set or the importance function have not been properly chosen. In case of violation of the former, in Section \ref{ss:choice_AIF} we proposed a test for the appropriateness of the choice of the set $A$. Additional verification can be performed to assess whether resampling from the set $S_{\rm rec}$ obtained in Step \ref{step:aA} of the RMS algorithm \dc{is a good approximation of taking i.i.d. samples} of $X^A_1$. This implies that $\hat\a_A$ and $\hat T_B$ are independent, but if they are independent then necessarily
\begin{equation}\label{eq:RE_calc}
\RE^2(\hat\gamma) = \RE^2(\hat\a_A) + \RE^2(\hat T_B) \, .
\end{equation}
\dc{Thus, if (\ref{eq:RE_calc})} is not approximately satisfied, it is an indication that $S_{\rm rec}$ \dc{does not represent the distribution of $X_1^A$ well}. We emphasize that the relative error of $\hat\gamma$ presented in \magenta{the} tables is calculated as in \eqref{def:sampleRE}.

%------------------------------------------------------------------------

\subsection{Ornstein-Uhlenbeck Process}\label{ss:OU}

Let $(X_t)_{t\geq0}$ be a \magenta{$d$}-dimensional Ornstein-Uhlenbeck process ($d$-dim OU), i.e., a process taking values in $\R^d$ \magenta{solving the SDE}
\begin{equation}\label{def:OU_sde}
\d X_t = - Q X_t\d t + \d W_t
\end{equation}
with $Q\in\R^{d\times d}$ and $(W_t)_{t\geq0}$ denoting a standard $d$-dimensional Wiener process. \magenta{Applying the explicit Euler numerical scheme  to \eqref{def:OU_sde}}, with time step $h>0$ yields
\begin{equation}\label{def:OU_euler}
X_{n+1} = (I - Qh)X_n + Z_n,
\end{equation}
with \magenta{$I$ the $d$-dimensional identity matrix $I$}, and \magenta{$Z_1,Z_2,\ldots$ i.i.d.\ $d$-dimensional standard normal random variables}. It is known \cite{schurz1999invariance} that the stationary distribution $\mu$ of \eqref{def:OU_euler} exists if there exists a positive-definite matrix $M = (M_{ij})_{i,j\in\N}$ solving
\begin{equation}\label{eq:M}
M = (I-Qh)M(I-Qh)^\top + hI;
\end{equation}
then the stationary distribution $\mu$ is $d$-dimensional \magenta{centered} normal with covariance matrix $M$. The rare event of our interest is \dc{the exceedance of a high threshold in the first dimension under} the stationary distribution (of the discrete-time Markov chain  in \eqref{def:OU_euler}), as in \eqref{def:gamma_u}. Eq.\ \eqref{eq:M} is a well-known Sylvester equation and its solution $M$ can be found numerically, \magenta{so that $\gamma(u)$ can be evaluated as}
\begin{equation}\label{eq:gamma_OU_explicit}
\gamma(u) = \Phi(-u/\sqrt{M_{11}}),
\end{equation}
with $\Phi(\cdot)$ the standard normal cdf. Knowing \magenta{the ground truth} $\gamma(u)$ gives us means to determine how accurate the RMS estimator $\hat\gamma$ is.

\vb

In the following three subsections we study the OU process with different sets of parameters but with the same choice of the recurrency set and importance function, as in \eqref{def:A_u} and \eqref{def:IF_u}. First, we study the simplest case of a one-dimensional OU process. \magenta{This is an `ideal' example in the sense that} Assumptions \ass\ are \dc{(approximately)} satisfied. Second, we study a multidimensional OU process; while the simplifying assumptions do not seem to be satisfied, they are `close enough' \dc{for} the RMS method to  give satisfactory results. The third case describes a two-dimensional OU process with the matrix $Q$ chosen such that Assumptions \ass\ are not satisfied \dc{\magenta{for} our choice of \magenta{the} recurrency set and \magenta{the} importance function}.

%-------------------------------------------------------------------

\subsubsection{1-dim OU}\label{sss:OU1}

In this experiment we put $d=1$, $Q=1$, $h=0.01$. The recurrency set $A(\ell)$ and importance function $H(\cdot)$ are as in \eqref{def:A_u} with $\ell=0$ and \eqref{def:IF_u} respectively.

\dc{If we would study the stationary distribution of the original SDE driven by \eqref{def:OU_sde} (rather than the time-discrete numerical solution in \eqref{def:OU_euler}), \magenta{then} the paths of the process would be continuous and thus $X^A_1 = 0$ a.s. Moreover, because of their continuity, these paths must cross all intermediate states $x\in(0,u)$ before reaching $B$. Therefore $x\mapsto \p_x(\tau_B<\tau^\text{in}_A)$ is an increasing function, implying that the distance-based importance function satisfies \eqref{ass:amrein} in the continuous-time case. By similar arguments, $X_{\tau_B} = u$ a.s., and hence \eqref{ass:trare} is satisfied as well in that case.}

%Consider for a moment that we study the stationary distribution of the original SDE driven by \eqref{def:OU_sde} and not the time-discrete numerical solution in \eqref{def:OU_euler}. The paths of the original process are continuous, thus $X^A_1 = 0$ a.s. What is more, in order to reach set $B = [u,\infty)$, the path must first cross all intermediate states $x\in(0,u)$ so $x\mapsto \p_x(\tau_B<\tau^\text{in}_A)$ is an increasing function. This implies that distance-based importance function satisfies \eqref{ass:amrein}. Moreover, the continuity argument also implies $X_{\tau_B} = u$ a.s.; since there is only one possible entry state to the rare event, \eqref{ass:trare} follows.

The Markov chain driven by \eqref{def:OU_euler} is a \magenta{discrete-time} approximation of \eqref{def:OU_sde}, so the assumptions  will not be satisfied \textit{exactly}. In particular, \dc{we note} that for any time step $h>0$, the support of $X_{\tau_B}$ is \dc{the entire} halfline $[u,\infty)$ because in principle the process can exceed the threshold $u$ by any positive value upon the first entry. This shows that Assumption \eqref{ass:trare} is not satisfied. \magenta{An} analogous argument can be used to show that  \magenta{Assumption} \eqref{ass:amrein} is not satisfied \magenta{either}. Nonetheless, for \magenta{a} small time step $h>0$, extreme overshooting upon the first entry \dc{(i.e., $X_{\tau_B}$ \magenta{being} significantly larger than $u$, or $X_{\tau_k}$ significantly larger than $\ell_k u$) is very unlikely. \magenta{We conclude that the} assumptions are} satisfied approximately.

Since the value of $\gamma(u)$ can be evaluated using \eqref{eq:gamma_OU_explicit}, we chose the thresholds $u$ to match the desired value of $\gamma(u)$, as in Table \ref{tab:1dimOU}. The results show that $\RE(\hat T_B) \approx 5\cdot10^{-3}$, as desired in \eqref{eq:desired_RE}; this is a good indication that Assumptions \ass\ are satisfied. Also, the relative error calculated under the independence assumption via \eqref{eq:RE_calc} matches the estimated $\RE(\hat\gamma)$.

\textit{Conclusions.} \magenta{In this setting} the RMS algorithm is very efficient, as compared with MC. The numerical results agree very well with \magenta{the theoretical outcomes}, \magenta{confirming our observation that Assumptions \ass\ are approximately satisfied.}

\begin{table}[tb]
\centering
\begin{tabular}{c|ccccc}
$\gamma(u)$        &  $10^{-3}$  & $10^{-4}$  & $10^{-5}$ & $10^{-6}$  & $10^{-7}$  \\
\hline
$\hat\gamma$ & $9.94\cdot 10^{-4}$ & $9.93\cdot 10^{-5}$ & $9.96\cdot 10^{-6}$ & $9.96\cdot 10^{-7}$ & $9.96\cdot 10^{-8}$ \\ 
$\RE(\hat\gamma)$ & 3.95e-03 & 5.45e-03 & 6.53e-03 & 6.31e-03 & 5.49e-03 \\ 
\hline\hline 
$\Eff(\hat\gamma)$ & 4.1 & 8.9 & 45.2 & 378.9 & 1836.2 \\ 
$\RE(\hat T_B)$ & 3.90e-03 & 4.99e-03 & 6.42e-03 & 6.30e-03 & 5.32e-03
\end{tabular}
\caption{RMS algorithm for an 1-dim OU process. Parameters: $Q=1$, $A = \{x_1 \leq 0\}$, $B = \{x_1\geq u\}$; $u$ has been chosen using \eqref{eq:gamma_OU_explicit} to match the values of $\gamma$ in the first row. We have $\hat\a_A = 0.0225$ and $\RE(\hat\a_A) = 1.66\cdot10^{-3}$.
	}\label{tab:1dimOU}
\end{table}

%----------------------------------------------------------------------

\subsubsection{10-dim OU, $Q$ with real eigenvalues}\label{sss:OUd}

In this experiment we put $d=10$, $h=0.01$. The matrix $Q = (Q_{ij})_{ i,j\in\{1,\ldots,d\}}$ is randomly generated such that all its eigenvalues are real. The recurrency set $A(\ell)$ and importance function $H(\cdot)$ are as in \eqref{def:A_u} with $\ell=0$ and \eqref{def:IF_u} respectively. 

In Fig.\ \ref{fig:10dimOU_path_hit_18} we plot four randomly chosen recurrency cycles, projected onto the first and second dimension, which have reached the rare event $B$. These conditional paths seem to follow a linear pattern; similar behavior is \dc{seen in other projections (not shown)}. This indicates that attaining high values in the first dimension is coupled with attaining high values in the second dimension (and similar statements can be made about other dimensions). Therefore, the distance-based importance function is not expected to satisfy \eqref{ass:amrein}, as it does not take this behavior into account; an ideal importance function should give larger importance to states which attain \textit{simultaneously} high values in the first and second dimension. While the distance-based importance function is not the most appropriate choice, it is still expected to give satisfactory results, as it \magenta{drives} the paths gradually towards the rare event.

The results of the RMS algorithm are presented in Table~\ref{tab:10dimOU}. It can be seen that the values of $\RE(\hat T_B)$ \dc{do not} exactly match the desired value $5\cdot10^{-3}$ in \eqref{eq:desired_RE}, which in view of the earlier discussion is not surprising, as we did not expect Assumptions \ass\ to hold. However, \dc{the estimates $\hat\gamma$ are still very accurate, and}
\magenta{the efficiency is still excellent (relative to the MC method)}.

\vb

\textit{Conclusions.} \magenta{This experiment shows that the RMS algorithm can be effectively implemented in a multidimensional setting}, even when Assumptions \ass\ are violated. This underscores the robustness of the distance-based importance function.
%{\color{green} [I WOULD LEAVE THIS OUT - DC]: in a setting, in which the paths gradually drift towards the rare event.}

\begin{figure}[tb]
        \centering
       \includegraphics[width=.7\linewidth]{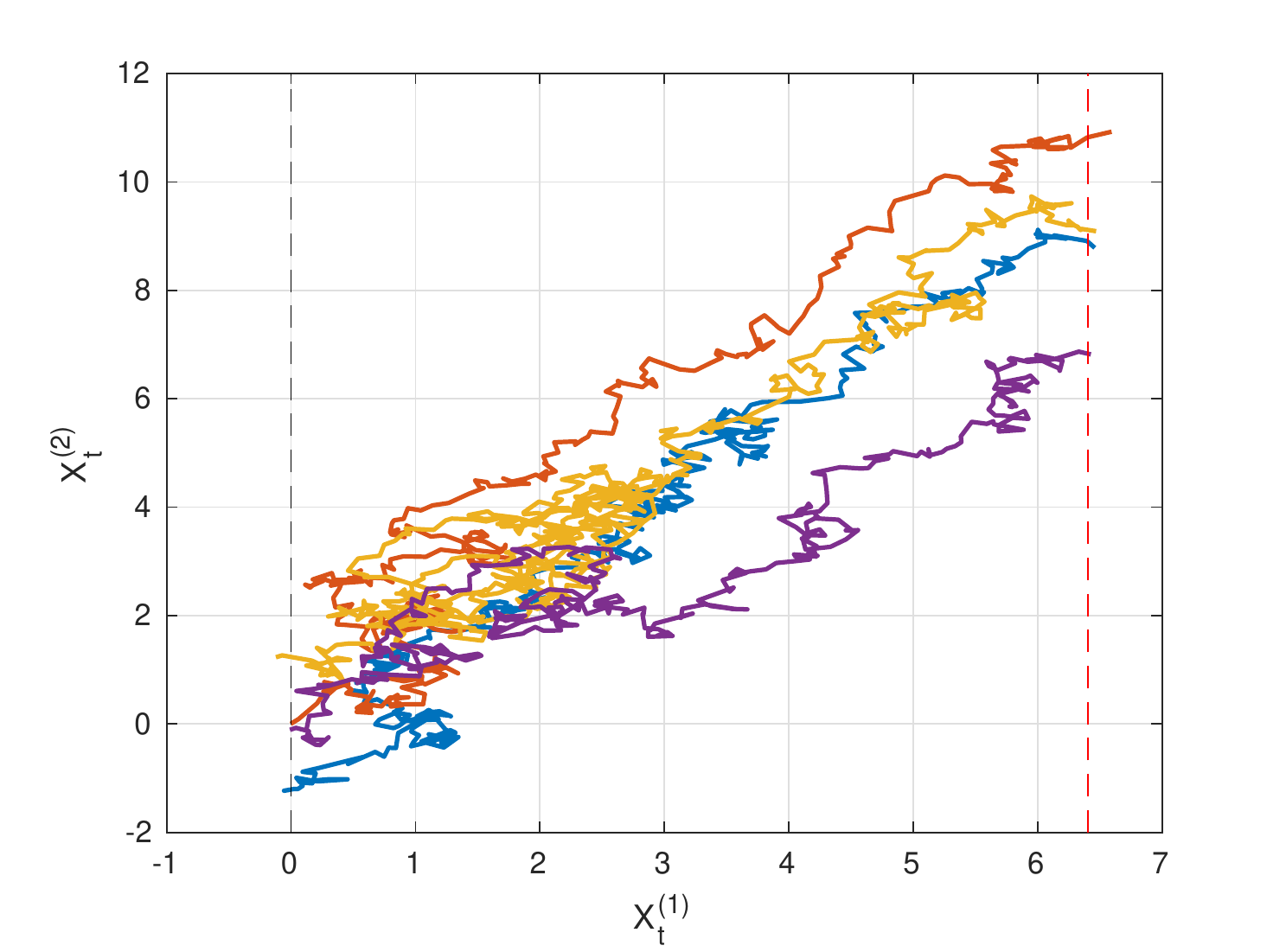}
      \caption{10-dim OU process. Four random realizations of recurrency cycles conditioned on reaching the rare set. The cycles have been plotted until the first hitting time of $B$. Parameters: $A=\{x_1\leq 0\}$, $B = \{x_1\geq u\}$ with $u\approx 6.4$ such that and $\gamma(u)=10^{-6}$.
            }\label{fig:10dimOU_path_hit_18}
\end{figure}

\begin{table}[tb]
\centering
\begin{tabular}{c|ccccc}
$\gamma(u)$        &  $10^{-3}$  & $10^{-4}$  & $10^{-5}$ & $10^{-6}$  & $10^{-7}$  \\
\hline
$\hat\gamma$ & $1.00\cdot 10^{-3}$ & $9.95\cdot 10^{-5}$ & $1.02\cdot 10^{-5}$ & $9.92\cdot 10^{-7}$ & $1.00\cdot 10^{-7}$ \\ 
$\RE(\hat\gamma)$ & 7.84e-03 & 1.03e-02 & 1.35e-02 & 1.12e-02 & 1.49e-02 \\ 
\hline\hline 
$\Eff(\hat\gamma)$ & 0.8 & 2.4 & 9.3 & 34.9 & 180.5 \\ 
$\RE(\hat T_B)$ & 7.87e-03 & 1.02e-02 & 1.35e-02 & 1.12e-02 & 1.49e-02 
\end{tabular}
\caption{RMS algorithm for a 10-dim OU process. Parameters: $Q$ is a matrix with only real eigenvalues, $A = \{x_1 \leq 0\}$, $B = \{x_1\geq u\}$; $u$ has been chosen using \eqref{eq:gamma_OU_explicit} to match the values of $\gamma$ in the first row. We have $\hat\a_A = 0.0124$, $\RE(\hat\a_A) = 2.46\cdot10^{-3}$.}\label{tab:10dimOU}
\end{table}

%------------------------------------------------------------

\subsubsection{2-dim OU, $Q$ with complex eigenvalues}\label{sss:OU2}

In this experiment we put $d=2$, $h=0.01$. We choose $Q$ to have non-real \dc{eigenvalues}: \magenta{for a positive $\theta$,}
\begin{equation}\label{eq:Q_theta}
Q(\theta) = \Bigg[ \, \begin{matrix}1 & \theta \\ -\theta & 1\end{matrix} \ \Bigg] \, .
\end{equation}
\dc{The drift generates a rotating (or spiraling) motion of the paths, with the speed of rotation increasing as \magenta{$\theta$} increases. We}
compare the efficiency of the RMS method for increasing values of $\theta$. The recurrency set $A(\ell)$ and importance function $H(\cdot)$ are as in \eqref{def:A_u} with $\ell=0$ and \eqref{def:IF_u} respectively.
 
The results are presented in Table \ref{tab:2dimOU}. We see that for most values of $\theta$, RMS outperforms the Monte Carlo, but the larger $\theta$ is, the \dc{lower the efficiency ratio Eff($\hat\gamma$) becomes}. At the same time, as $\theta$ grows, the value of $\RE(\hat T_B)$ deviates more and more from the desired target $5\cdot10^{-3}$, as in \eqref{eq:desired_RE}. This indicates a violation of Assumptions \ass. \dc{We note that the estimates $\hat\gamma$ are quite accurate nonetheless, with a minor relative error of a few percent visible for larger values of $\theta$.}

In Fig.\ \ref{fig:2dimOUpath_rho3} we plot five random recurrency cycles conditioned on reaching the rare set $B$. We see that the paths do not gradually drift towards $B$, \dc{\magenta{but rather}  first move far away from $B$, due to the drift-induced rotation. This} hints that the distance-based importance function might be a poor choice. \magenta{Fig.\ \ref{fig:2dimOU_Sreg_quantile} shows that even \magenta{property} \eqref{ass_A} seems to be violated. In this figure we compare the histograms of $S_\text{\rm rec}$ and $S^q_\text{\rm rec}$ in order to compare the distributions of $X^A_1$ and $X^A_+$ (see the discussion Section \ref{ss:choice_AIF}). {The figure shows  that $X^A_+$ has more probability mass in the sets  $\{x_2\leq-1\}$ or $\{x_2\geq1\}$ than $X^A_1$. }}

\vb

\textit{Conclusions.} When $Q$ has non-real eigenvalues, the \magenta{na\"{\i}ve} choice of \magenta{the} recurrency set and \magenta{the} distance-based importance function (i.e., \eqref{def:A_u} and \eqref{def:IF_u}) seems inadequate and leads to \magenta{a relative error higher than expected}. This underscores the fact that one has to be careful with the choice of $A$ and $H(\cdot)$ and verify whether Assumptions \ass\ are satisfied; \magenta{this can be done e.g.\ by} the means described in Section \ref{ss:choice_AIF}. Despite violation of Assumptions \ass, RMS still \dc{gives rather accurate estimates of $\gamma$, and} outperforms Monte Carlo for small $\theta$. 

\begin{figure}[tb]
 \centering
       \includegraphics[width=0.7\linewidth]{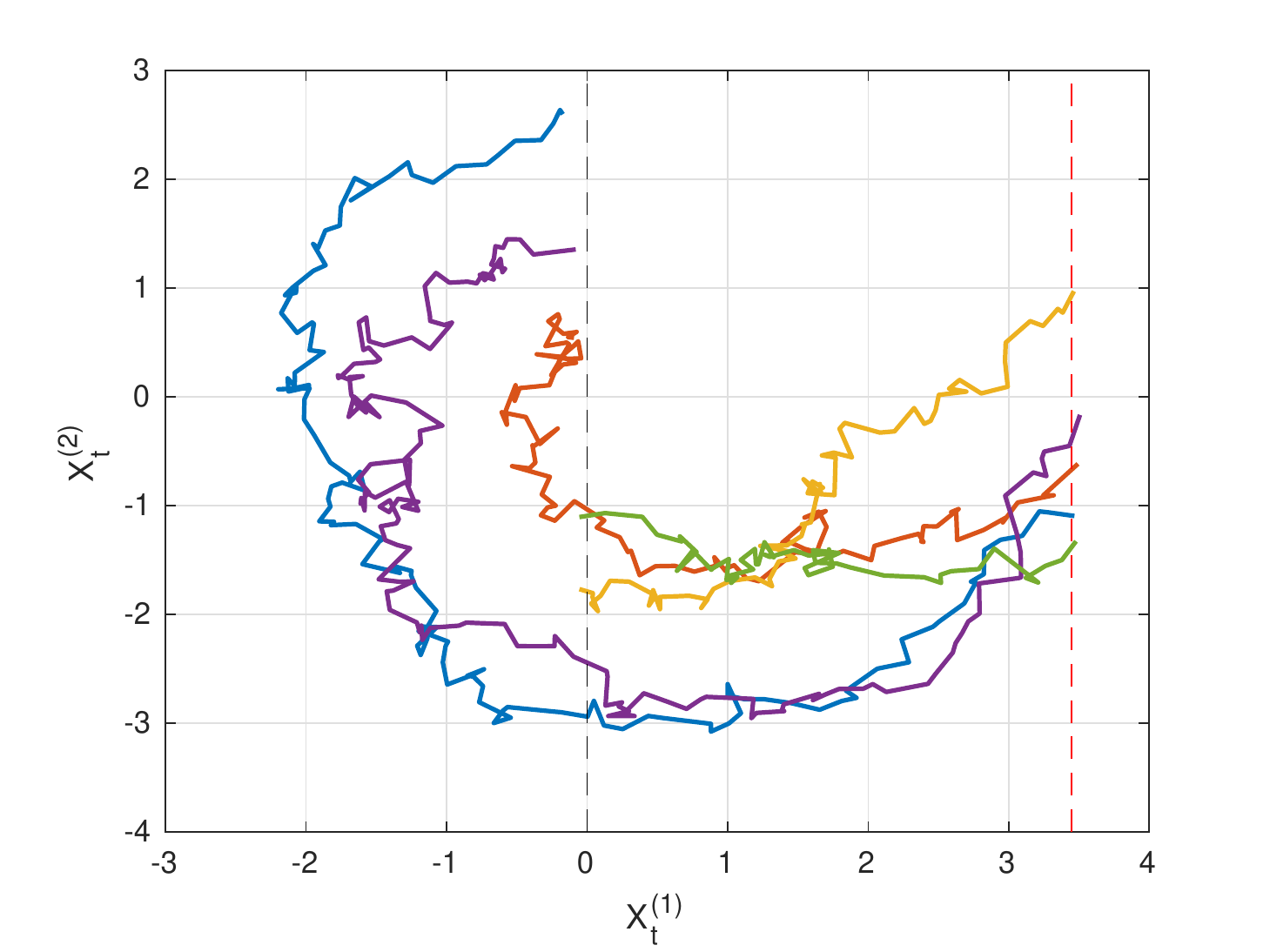}
      \caption{2-dim OU process. Five random realizations of recurrency cycles conditioned on reaching the rare set. The cycles have been plotted until the first hitting time of $B$. Parameters: $A=\{x_1\leq 0\}$, $\theta=3$, $B = \{x_1\geq u\}$ with $u\approx 3.4$ such that $\gamma(u)=10^{-6}$. 
      }\label{fig:2dimOUpath_rho3}
\end{figure}

\begin{figure}[tb]
 \centering
       \includegraphics[width=0.7\linewidth]{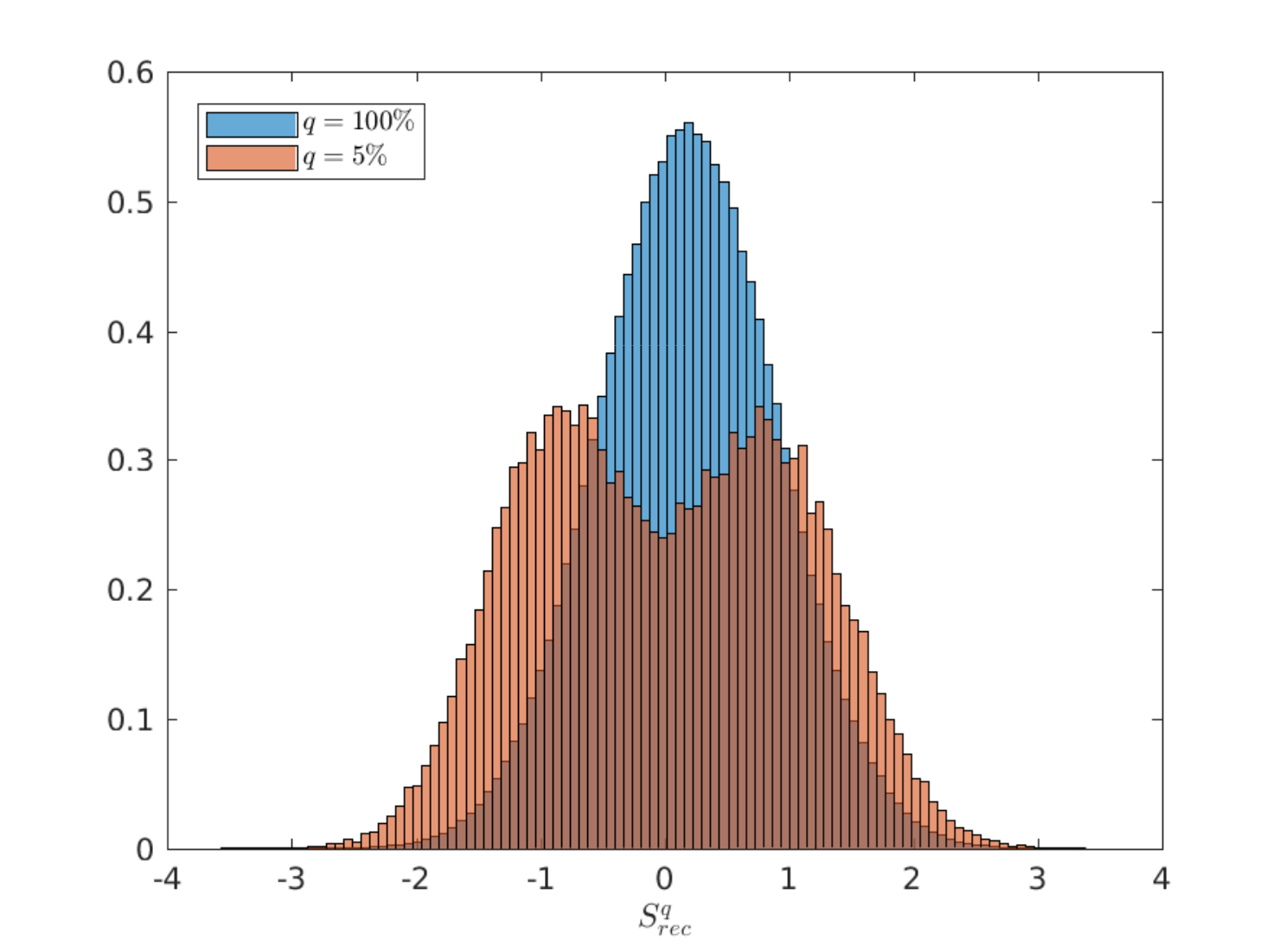}
      \caption{2-dim OU process, $\theta=3$. Marginal histograms of $S^q_\text{rec}$ projected onto the second dimension. The histograms have been normalized to a probability density function.
      }\label{fig:2dimOU_Sreg_quantile}
\end{figure}

\begin{table}[tb]
\centering
\begin{tabular}{c|ccccc}
$\theta$             & 0.5 & 1 & 1.5 & 2 & 3        \\
\hline
$\hat\gamma$ & $9.91 \cdot 10^{-7}$ & $1.00 \cdot 10^{-6}$ & $1.00 \cdot 10^{-6}$ & $9.73 \cdot 10^{-7}$ & $9.60 \cdot 10^{-7}$ \\ 
$\RE(\hat\gamma)$ & 8.20e-03 & 1.05e-02 & 2.34e-02 & 2.66e-02 & 4.01e-02 \\ 
\hline\hline 
$\Eff(\hat\gamma)$ & 31.9 & 27.9 & 7.1 & 5.8 & 1.0 \\ 
$\RE(\hat T_B)$ & 7.63e-03 & 1.05e-02 & 2.37e-02 & 2.67e-02 & 4.01e-02
\end{tabular}
\caption{RMS algorithm applied to 2-dim OU process. Parameters: $Q(\theta)$ as in \eqref{eq:Q_theta}, $A = \{x_1\leq0\}$, $B=\{x_1\geq u\}$; $u$ has been chosen depending on $\theta$ such that in every case $\gamma(u)=10^{-6}$.}\label{tab:2dimOU}
\end{table}

%-------------------------------------------------------------------

\subsection{Franzke (2012) Stochastic Climate Model}\label{ss:franzke}

As our final example, we consider the low-order stochastic climate model presented by \citet{franzke2012predictability}. This is a 4-dimensional  SDE with certain key features that are also present in more complex climate models, including nonlinear (quadratic) drift terms that are energy-conserving. We refer to  \citet{franzke2012predictability} for a more detailed discussion of the physical interpretation of this model.

The model is given by the following set of SDEs. \magenta{It uses a standard, two-dimensional Wiener process $(W\upot,W\uptt)$. We write $x_i:=X^{(i)}_t$, $y_i:=Y^{(i)}_t$ and $W_i := W^{(i)}_t$ to simplify notation. We consider the system}
\begin{align*}
& \begin{aligned}
\d x_1 & = \mu \big( -x_2(L_{12} + a_1 x_1 + a_2 x_2) + d_1x_1 + F_1 \\
& \qquad + L_{13}y_1 + B_{123}^1 x_2 y_1 + (B_{131}^2 + B_{113}^2)x_1 y_1 \big) \d t
\end{aligned}\\
& \begin{aligned}
\d x_2 & = \mu \big( +x_1(L_{21} + a_1 x_1 + a_2 x_2) + d_2x_2 + F_2 \\
& \qquad + L_{24}y_2 + B_{213}^1 x_1 y_1 + (B_{242}^3 + B_{224}^3)x_2 y_2 \big) \d t
\end{aligned}\\
& \d y_1 = \mu \big( -L_{13}x_1 + B_{312}^1 x_1 x_2 + B_{311}^2 x_1^2 + F_3 - \tfrac{\gamma_1}{\ep} y_1 \big)\d t + \tfrac{\sigma_1}{\sqrt\ep} \d W_1 \\
& \d y_2 = \mu \big( -L_{24}x_2 + B_{422}^3 x_2 x_2 + F_4 - \tfrac{\gamma_2}{\ep} y_2 \big)\d t + \tfrac{\sigma_2}{\sqrt\ep} \d W_2
\end{align*}
When the parameter $\ep$ is set to a small value, a separation of timescales is created between the  variables $x_1,x_2$ (slow) and $y_1,y_2$ (fast). The main interest is in the behavior of the slow variables $x_1,x_2$.

\magenta{The parameters we use match those used in \citet{franzke2012predictability}. This means that we set
$\mu=1$, the $B$-coefficients are given by $B_{123}^1=4$, $B_{213}^1=4$, $B_{312}^1=-8$, $B_{131}^2=0.25$, $B_{113}^2=0.25$, $B_{311}^2=-0.5$, $B_{242}^3=-0.3$, $B_{224}^3=-0.4$, $B_{422}^3=0.7$, the $L$-coefficients by $L_{13}=-L_{24}=-0.2$,  and the other parameters by $\omega=1$, $a_1=1$, $a_2=-1$, $d_1=-0.2$, $d_2=-0.1$, $\gamma_1=\gamma_2=1$, $\sigma_1=3$, $\sigma_2=1$. In addition  we put $L_{12} = -L_{21} = 1, \ep = 0.2$. The forcing vector $(F_1,F_2,F_3,F_4)$ is given by $(-0.25,0,0,0)$.}

Since this process is non-standard, in order to build intuition, we \magenta{first generated} a contour plot of the estimated stationary density of $(x_1,x_2)$; see Fig.\ \ref{fig:franzke_heatmap}. The process \magenta{turns out to} randomly switch between two modes: one mode with $x_1\leq x_2$ and a second mode with $x_1\geq x_2$. The estimated density function in Fig.\ \ref{fig:franzke_heatmap} shows that the process is more likely to be in the second mode.

\vb

We use the explicit Euler scheme with $h_0=10^{-4}$ but we store the values of the process every $h=0.01$. The small integration time step $h_0$ is needed for numerical stability. Similar to the previous examples, the rare event we study is the exceedance of a high threshold by $x_1$ under the stationary distribution, \magenta{cf.} \eqref{def:gamma_u}. We choose the recurrency set $A(\ell)$ as in \eqref{def:A_u} with $\ell^*=7.9$ suggested by the algorithm \eqref{eq:lstar}. The importance function $H(\cdot)$ is as in \eqref{def:IF_u}.

\vb

The results of the RMS method are outstanding, see Tab.\ \ref{tab:franzke}. For $u=18.5$, when $\gamma(u)\approx10^{-7}$, \dc{we find Eff($\hat\gamma$)$\approx$ 1522. In other words,} the RMS algorithm is more than 1500 times faster than MC. The values of $\RE(\hat T_B)$ match the desired $5\cdot10^{-3}$ (see \eqref{eq:desired_RE}) very closely even for very high thresholds, indicating that Assumptions \ass\ are satisfied. A random realization of a cycle reaching the rare event, shown in Fig.\ \ref{fig:franzke_pathrare}, is yet another indication that the distance-based importance function is a good choice, as the path seems to gradually drift towards the rare event.

\vb

\textit{Conclusions.} This example shows a \magenta{successful} application of the RMS algorithm to a multidimensional nonlinear stochastic-dynamical model with characteristics of complex climate models. We find that RMS is up to three orders of magnitude faster than MC in this example, and the efficiency gain is expected to be even larger for higher thresholds \magenta{$u$}.

\begin{figure}[tb]
        \centering
       \includegraphics[width=0.7\linewidth]{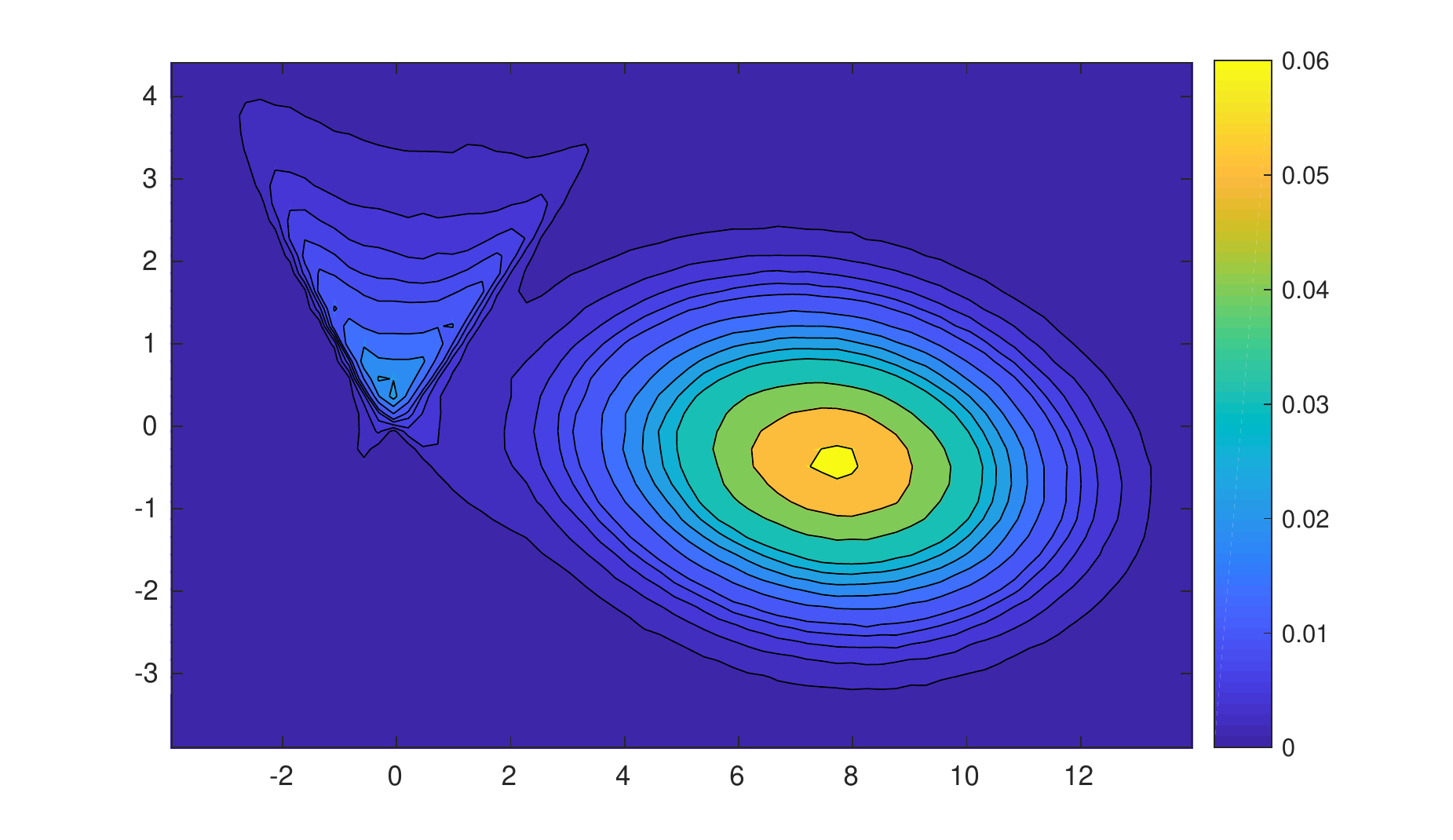}
      \caption{Contour plot of the marginal stationary density of slow variates $(x_1,x_2)$ of the model of \citet{franzke2012predictability}.
      }\label{fig:franzke_heatmap}
\end{figure}

\begin{figure}[tb]
        \centering
       \includegraphics[width=0.7\linewidth]{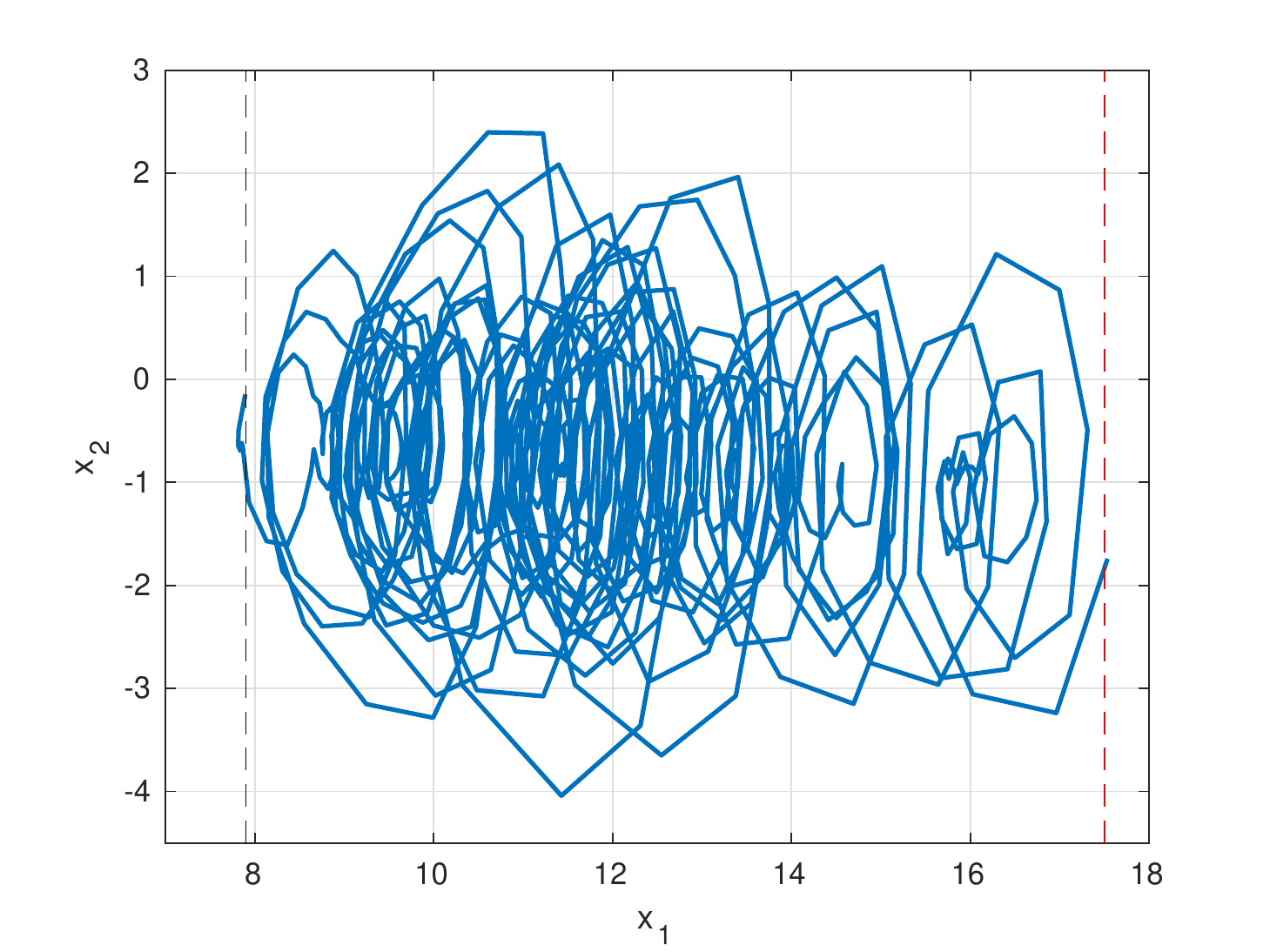}
      \caption{The model of \citet{franzke2012predictability}. A random realization of a recurrency cycle conditioned on reaching the rare set. The cycle has been plotted until the first hitting time of $B$. Parameters: $A=\{x_1\leq \dc{7.9}\}$, $B = \{x_1\geq 17.5\}$, $\gamma(17.5)\approx 1.14\cdot10^{-6}$.
      }\label{fig:franzke_pathrare}
\end{figure}

\begin{table}[tb]
\centering
\begin{tabular}{c|ccccc}
$u$             & 14 & 15 & 16 & 17.5 & 18.5        \\
\hline
$\hat\gamma$ & $1.08 \cdot 10^{-3}$ & $1.99\cdot 10^{-4}$ & $3.00 \cdot 10^{-5}$ & $1.14\cdot 10^{-6}$ & $9.78\cdot 10^{-8}$ \\ 
$\RE(\hat\gamma)$ & 6.1e-03 & 7.2e-03 & 7.4e-03 & 7.4e-03 & 5.8e-03 \\ 
\hline
$\hat\gamma^{\rm MC}$ & $1.08 \cdot 10^{-3}$ & $2.00 \cdot 10^{-4}$ & $2.98 \cdot 10^{-5}$ & $1.12 \cdot 10^{-6}$ & $8.85 \cdot 10^{-8}$ \\ 
$\RE(\hat\gamma^{\rm MC})$ & 1.4e-03 & 2.9e-03 & 6.5e-03 & 2.7e-02 & 8.5e-02 \\ 
\hline\hline
$\Eff(\hat\gamma)$ & 1.9 & 8.6 & 32.1 & 269.9 & 1521.8 \\ 
$\RE(\hat T_B)$ & 5.1e-03 & 6.4e-03 & 7.2e-03 & 6.6e-03 & 5.4e-03 
\end{tabular}
\caption{RMS algorithm applied to the model of \citet{franzke2012predictability}. Parameters: $A=\{x_1 \leq 7.9\}$, $B=\{x_1>u\}$. We have $\hat\a_A = 0.0124$, $\RE(\hat\a_A) = 2.83\cdot 10^{-3}$.}\label{tab:franzke}
\end{table}

\section{Summary}\label{s:discussion}

In this manuscript we have proposed a new algorithm for \magenta{the} estimation of small steady-state probabilities $\gamma = \mu(B)$, as in \eqref{def:gamma}, \dc{of Markov processes with continuous state space. \magenta{Our approach, which we have called  the Recurrent Multilevel Splitting (RMS) algorithm,}} is based on {the} alternative representation (\ref{def:rec_prop}) of $\gamma$ (\magenta{as given in Theorem \ref{thm:reccurence_cycles}}). \dc{This representation is obtained by} dissecting the \magenta{path of the} Markov process into recurrency cycles, each cycle beginning with an inwards crossing of a set $A$. \dc{It allows to transform the problem of estimating $\gamma$ essentially into the problem of estimating  $T_B$,} the expected time spent in \magenta{the} set $B$ in a recurrency cycle.

\vb

In order to efficiently estimate $T_B$ we use Multilevel Splitting (MLS), but we emphasize that other rare event simulation methods could have been used instead (such as Genealogical Particle Analysis or Importance Sampling). We have derived optimal parameters for the MLS in Appendix \ref{appendix:optimal_parameters}, and we have shown (Theorem \ref{thm:logeff_rms}) that under simplifying assumptions, a \dc{suitable} choice of \magenta{the} recurrency set $A$ in combination with the optimal choice of \magenta{the} parameters leads to logarithmic efficiency of the RMS algorithm.

\vb

In Section \ref{s:experiments}, four numerical studies were presented, \dc{where we used the RMS algorithm to estimate steady state probabilities of high threshold exceedances for various SDEs discretized in time. The experiments demonstrate that RMS gives accurate results. \magenta{Furthermore,} they}
unanimously show the efficiency gain of RMS compared to Monte Carlo; in the most notable case of the \cite{franzke2012predictability} model (Section \ref{ss:franzke}), RMS outperforms MC by up to three orders of magnitude.

\vb

\dc{One of the numerical experiments (Section \ref{sss:OU2}) was designed to give suboptimal results, with an SDE displaying rotating motion so that the most straightforward choices of \magenta{the} recurrency set and importance function (as used in the experiments) were \magenta{expected to be not very suitable}. Although the estimates obtained with RMS were still quite accurate, the efficiency gain of RMS compared to MC was decreasing as the rotation speed was increasing. This example showed how the choice of \magenta{the} recurrency set and \magenta{the} importance function can impact the performance of the algorithm.}

\dc{In light of this example, an interesting topic for} future research is the choice of the recurrency set $A$. As already mentioned in Section \ref{ss:choice_AIF}, \magenta{a good choice of $A$ should be a suitable compromise between visiting $A$ relatively often and \eqref{ass_A} being (approximately) met}. We have proposed a method of optimizing $A(\ell)$ parametrized by $\ell$ in \eqref{eq:lstar}, and pointed out a method of testing whether $A$ satisfies \eqref{ass_A} through a quantile validation \eqref{def:Sq}. \dc{Further development of these ideas to construct an optimal $A$ is a challenging open research topic.}

\section*{Acknowledgments}
We thank the organizers of the {\it Summer School in Monte Carlo for Rare Events} (June 2016 at Brown University) for making lecture notes available. 
This work is part of the research programme `Mathematics of Planet Earth' which was funded by the Netherlands Organisation for Scientific Research (NWO), grant number 657.014.003. Michel Mandjes' research is partly funded by the NWO Gravitation Programme NETWORKS, grant number 024.002.003.

\appendix

\section{Technical Results}\label{appendix:technical_results}

\begin{proof}[Proof of Theorem $\ref{thm:reccurence_cycles}$]
Define a new Markov chain $Z_n := (X_{n-1},X_n)$; it is also positive Harris with a stationary measure $\tilde\mu$ satisfying, for measurable sets $C_0,C_1$,
\begin{align*}
\tilde\mu(C_0,C_1) = \p(X_0\in C_0, X_1\in C_1\mid X_0\sim\mu).
\end{align*}
We see that the stopping times $S_n$ coincide with the times the process $Z_n$ visits a set $\mathcal A := (A^c,A)$, with $A^c := \R^d\setminus A$. Since $\mu(A)\in(0,1)$ we have \[\a_A = \p_\mu(X_0\in A, X_1\in A^c)>0.\] According to \citet[Thm.\ 10.4.9]{meyn2012markov} we have, with $\tau_{\mathcal A} := \inf\{n>0: Z_n\in\mathcal A\}$,
\begin{align*}
\tilde\mu(\R^d,B) = \int_{\mathcal A} \tilde\mu(\d x,\d y)\,\Exp_x\sum_{n=0}^{\tau_{\mathcal A}-1} \ind \{Z_n\in (\R^d,B)\}.
\end{align*}
Due to $\tilde\mu(\R^d,B) = \mu(B)$, $\{Z_n\in (\R^d,B)\} = \{X_n\in B\}$, and $\tilde\mu(\mathcal A) = \a_A$, it follows that
\begin{align*}
\mu(B) & = \a_A \cdot\Exp\bigg( \sum_{n=0}^{\tau_{\mathcal A}-1} \ind \{X_n\in B\} \mid Z_0\sim\tilde\mu, Z_0\in\mathcal A\bigg).
\end{align*}
Finally, we recognize that the conditioning above is equivalent to $X_0$ being distributed as an initial point of a recurrency cycle $X^A_1$ in stationarity, so that we conclude \eqref{def:rec_prop}. Similarly, one can show that $\alpha_A = (\Exp_\mu L_1)^{-1}$ by considering the expected time spent in $(\R^d,\R^d)$ within a recurrency cycle.
\end{proof}

\begin{proof}[Derivation of \eqref{eq:SRE_Tb}]
Notice that \eqref{ass:amrein} implies that the number of times the $k$-th threshold $r_k$ is hit, is distributed as a sum of $n_{k-1}\,r_{m-1}$ \textit{independent} Bernoulli trials, each with probability of success $p_k$:
\begin{equation}\label{eq:amrein_implication}
\big(r_k\mid r_{k-1}\big) \eqd \text{Bin}(n_{k-1}r_{k-1},p_k);
\end{equation}
here $\text{Bin}(n,p)$ denotes a Binomial distribution with $n$ trials with success probability $p$, with the convention that $\text{Bin}(0,p) \equiv 0$.
Similarly, \eqref{ass:trare} implies that the total time spent in the rare set is distributed as a sum of $n_mr_m$ \textit{independent} copies from the distribution $R_+$:
\begin{equation}\label{eq:trare_implication}
\big(r_{m+1}\mid r_m\big) \eqd \sum_{k=1}^{n_mr_m} R^{(k)}_+,
\end{equation}
where $R^{(1)}_+,R^{(2)}_+,\ldots$ are i.i.d.\ copies of $R_+$ (with the empty sum being defined as 0).
Using \eqref{eq:amrein_implication} and the law of total variance we obtain, for $k\in\{1,\ldots m\}$, 
\begin{align*}
\Var(r_k) & = \Exp(\Var(r_k|r_{k-1})) + \Var(\Exp(r_k|r_{k-1})) \\
& = \Exp(n_{k-1}r_{k-1}p_k(1-p_k)) + \Var(n_{k-1}r_{k-1}p_k) \\
& = n_{k-1}p_k(1-p_k)\Exp (r_{k-1}) + n_{k-1}^2p_{k}^2 \Var(r_{k-1}).
\end{align*}
Similarly, using \eqref{eq:trare_implication} we obtain
\begin{align*}
\Var(r_{m+1}) = n_m\Exp(r_m)\Var(R_+) + n_{m}^2(\Exp R_+)^2 \Var(r_m).
\end{align*}
Combining these results with \eqref{eq:unbiased_Pk} yields \eqref{eq:SRE_Tb}.
\end{proof}

\section{Derivation of Optimal Parameters}\label{appendix:optimal_parameters}
Following \cite{amrein2011variant}, we assume that the computational effort $w_k$ in the $k$-th stage of Algorithm \ref{alg:MLS} (to sample a path starting from $X_{\tau_k}$ until $\min\{\tau_{k+1},\tau^\text{in}_A\}$) does not depend on the entry state $X_{\tau_k}$. Simplifying this further, we assume that $w_k$ does not depend on $k$, so without loss of generality,
\begin{equation}\label{ass:workload}
w_k \equiv 1, \ \ k\in\{0,\ldots, m\}.
\end{equation}
\redd{A more general cost $w_k$ can be considered for particular problems, see e.g. \cite{lagnoux2006rare}.}

Let $N_k := n_kr_k$,  for $k\in\{0,\ldots, m\}$, be the number of paths simulated in the $k$-th stage of the algorithm, with $r_0 := 1$. Then the average total workload equals
\begin{equation}\nonumber
W := \sum_{k=0}^m \Exp N_k,
\end{equation}
and since $\Exp r_k = p_1\cdots p_k$, cf.\ \eqref{eq:unbiased_Pk}, we conclude
\[\Exp N_k = \prod_{j=0}^k n_jp_j.\] Finally, we formulate the minimization problem
\begin{align*}
\text{minimize} \ \ & \ W := {\textstyle \sum_{k=0}^{m} \prod_{j=0}^k n_jp_j} \\
\text{with respect to:} \ \ & \ m,p_1, \ldots, p_m,n_0,\ldots,n_m \\
\text{subject to:} \ \ & \ \begin{cases}
\RE^2(\hat T_B) \leq \rho ,\\
\prod_{k=1}^m p_k = p_B, \\
m\in\N, \\
p_k \in(0,1), \ \ k\in\{1,\ldots, m\} ,\\
n_k\in\N, \ \ k\in\{0,\ldots, m\}.
\end{cases}
\end{align*}
In our simplified setting, i.e., under Assumptions \ass, we have derived a formula for the corresponding squared relative error in \eqref{eq:SRE_Tb}. We are able to solve the optimization problem above under \magenta{the} additional relaxation that the $n_k$ and $m$ are real and positive. To this end, it is helpful to denote
\begin{align*}
c_k & := {\textstyle \prod_{j=0}^{k-1}} n_jp_j, \ \ k\in\{1,\ldots,m+1\}, \\
a_k & := (1-p_k)/p_k, \ \ k\in\{1,\ldots m\}, \\
a_{m+1} & := \RE^2(R_+).
\end{align*}
Then we can write
\begin{align*}
W = \sum_{k=1}^{m+1} c_k \qand \RE^2(\hat T_B) = \sum_{k=1}^{m+1} \frac{a_k}{c_k}.
\end{align*}
We want to minimize the workload \magenta{$W$} under \magenta{the} constraint that $$\RE^2(\hat T_B) \leq \rho.$$ We do this in steps. First, we fix $m$ and the conditional probabilities $p_1,\ldots,p_m$, so that $a_1,\ldots,a_{m}$ are fixed (recall that $a_{m+1}$ is not a parameter of the algorithm). We relax the problem and let the splitting factors $n_k$ be allowed to attain any real, positive value. This means that we wish to solve (over $c_1,\ldots,c_{m+1}>0$) 
\begin{align*}
\text{minimize} \ \ & \ W(c_1,\ldots,c_{m+1}) := {\textstyle \sum_{k=1}^{m+1} c_k} \\
%\text{with respect to:} \ \ & \ c_1, \ldots, c_{m+1} \\
\text{subject to:} \ \ & \ \begin{cases}
g(c_1,\ldots,c_{m+1}) := \sum_{k=1}^{m+1} \frac{a_k}{c_k} \leq \rho, \\
c_k > 0, \ \ k\in\{1,\ldots, m+1\}.
\end{cases}
\end{align*}
The corresponding Karush--Kuhn--Tucker conditions are
\begin{align*}
\begin{cases}
\nabla W + \mu \nabla g = 0, \\
\mu (g-\rho) = 0, \\
\mu \in [0,\infty).
\end{cases}
\end{align*}
with the gradient `$\nabla$' taken with respect to vector $(c_1,\ldots,c_{m+1})$. These are solved by
\begin{align*}
c_k := \frac{1}{\rho} \sqrt{a_k}\sum_{j=1}^{m+1}\sqrt{a_j},
\end{align*}
with the optimal workload \[W = \frac{1}{\rho}\bigg(\sum_{j=1}^{m+1}\sqrt{a_j}\bigg)^2.\] In the next step, we keep $m$ fixed and minimize over $a_1,\ldots,a_m$. Notice that $1+a_k = 1/p_k$, so that our minimization problem takes the form
\begin{align*}
\text{minimize:} \ \ & \ W(a_1,\ldots,a_{m}) := {\textstyle \tfrac{1}{\rho}\big(\sum_{k=1}^{m+1} \sqrt{a_k}\big)^2} \\
%\text{with respect to:} \ \ & \ a_1, \ldots, a_{m} \\
\text{subject to:} \ \ & \ \begin{cases}
h(a_1,\ldots,a_{m}) := \prod_{k=1}^{m} (1+a_k) = p_B^{-1} ,\\
a_k > 0, \ \ k\in\{1,\ldots, m\}.
\end{cases}
\end{align*}
Not surprisingly, this system is solved by \[a_1=\ldots=a_m = p_B^{-1/m}-1,\] so that the optimal intermediate probabilities coincide:
\begin{align*}
p_k = p_B^{1/m}, \ \ k\in\{1,\ldots, m\},
\end{align*}
with the optimal workload being
\[W(m) = \frac{1}{\rho}\left(m\sqrt{p_B^{-1/m}-1}+\sqrt{a_{m+1}}\right)^2.\] The final step is finding the optimal number of thresholds $m$. We see that the minimizer of $W(m)$ is also a minimizer of
\begin{equation}\nonumber
m\sqrt{\exp(-\log(p_B)/m) - 1}.
\end{equation}
Again, we relax this problem, allowing $m$ to be any real, positive number. \red{Finally, the optimal parameters are:
\begin{equation}\label{eq:optimal_parameters}
\begin{aligned}
m & = c\,|\log p_B| \\
p_k & = p_\text{opt} := \frac{2c-1}{2c} \approx 0.2032, \ \ k\in\{1,\ldots, m \}, \\
n_0 & = \frac{1}{\rho\sqrt{2c-1}}\cdot \bigg(\frac{c\,|\log p_B|}{\sqrt{2c-1}} + \RE(R_+)\bigg), \\
n_k & = 1/p_{k+1} = 1/p_\text{opt}, \ \ k\in\{1,\ldots,m-1\} ,\\
n_m & = \RE(R_+)\cdot \frac{2c}{\sqrt{2c-1}}.
\end{aligned}
\end{equation}
with $c \approx 0.6275$ solving $\exp(1/c) = 2c/(2c-1)$ and the optimal workload reads as in \eqref{eq:optimal_workload}.} Since $m, n_k$ must be integers, we propose to simply round the optimal parameters to the closest integer. \redd{A similar result (but without the last splitting stage, in which we estimate the time spent in the set $B$) has been presented in \cite[Example 3.2.]{lagnoux2006rare}.}

\section{Logarithmic Efficiency of the RMS Algorithm}\label{appendix:efficiency}

In this section we study the efficiency of the RMS method, in the asymptotic regime that the rare event probability \eqref{def:gamma} tends to $0$ (i.e. $\gamma \to0$). First, we notice that if we fix the recurrency set $A$, then $\a_A$ does not change as $\gamma\to0$; \magenta{hence we  only have that $T_B\to0$}. \magenta{This indicates that asymptotic efficiency properties of RMS will be closely related to those of MLS.} 
In order to study \magenta{the performance of the estimator}, we first introduce \magenta{the} concepts of \textit{strong} and \textit{logarithmic efficiency}. 

Let $\hat\Psi_\ell$ be a family of unbiased estimators for $\Psi_\ell>0$, parametrized by $\ell$ such that $\Psi_\ell \to 0$, as $\ell\to\infty$. Let $W(\hat\Psi_\ell)$ denote the computation time corresponding to $\hat\Psi_\ell$. The estimator $\Psi_\ell$ is called \textit{strongly efficient} if
\begin{equation}\label{def:strongeff}
\limsup_{\ell\to\infty} \frac{W(\hat\Psi_\ell)\cdot\Var(\hat \Psi_\ell)}{\Psi_\ell^{2}} < \infty;
\end{equation}
and \textit{logarithmically efficient} if
\begin{equation}\label{def:logeff}
\lim_{\ell\to\infty} \frac{W(\hat\Psi_\ell)\cdot\Var(\hat \Psi_\ell)}{\Psi_\ell^{2-\ep}} = 0, \ \ \text{\rm for all } \ \ep>0.
\end{equation}
Strong efficiency implies that the workload needed to estimate the quantity of interest $\Phi_\ell$ with a desired accuracy \magenta{$\RE^2(\Psi_\ell) \le \rho$} is uniformly bounded as $\ell\to\infty$. Logarithmic efficiency implies that workload needed to achieve the accuracy $\RE^2(\Psi_\ell)=\rho$ is increasing slower than $\Psi_\ell^{-\ep}$ for any $\ep>0$, as $\ell\to\infty$. Evidently, strong efficiency implies logarithmic efficiency.

Before we prove the logarithmic efficiency of RMS in Theorem \ref{thm:logeff_rms} we show \magenta{an} inefficiency result for the Monte Carlo estimator for $T_B$. Let $\hat T^{\text{MC}}_B$ be a sample mean of $N$ independent copies of $R_1$. \magenta{We then have}
\begin{equation}
\RE^2(\hat T^{\text{MC}}_B) = \frac{1-p_B + \RE^2(R_+)}{p_B N}.
\end{equation}
Now to achieve a desired level of accuracy $\RE^2(\hat T^{\text{MC}}_B) \le \rho$, assuming \eqref{ass:workload}, the total required workload is
\begin{equation}\label{eq:MC_workload}
W(\hat T^\text{MC}_B) := \frac{1}{q}\cdot\frac{1-p_B + \RE^2(R_+)}{p_B}.
\end{equation}
As already noted in Section \ref{ss:simplified_setting}, $W(\hat T^{\rm MC}_B)$ is inversely proportional to $p_B$ and so it follows that the Monte Carlo estimator is not logarithmically efficient. 

We have seen, \magenta{cf.} \eqref{eq:optimal_workload}, that the workload of the MLS estimator with the optimal parameters $W(\hat T_B)$ is proportional to $(\log(p_B))^2$. It turns out that under mild additional assumption, the MLS algorithm is logarithmically efficient and thus so is RMS. We make this rigorous in the following theorem.

\begin{theorem}[Logarithmic Efficiency of RMS]\label{thm:logeff_rms}
Fix the recurrency set $A$ and let the set $B_\ell$ be parametrized by $\ell$, such that $\gamma_\ell := \mu(B_\ell)\to 0$ as $\ell\to\infty$. Assume 
\begin{itemize}
\item[$\circ$]
that \magenta{the} estimators $\hat\a_A$ and $\hat T_{B_\ell}$ are independent;
\item[$\circ$] that Assumptions \ass\ are valid for each $\ell$;
\item[$\circ$] that the workload satisfies \eqref{ass:workload};
\item[$\circ$] and that, for $\delta>0$ sufficiently small,
\begin{equation}\label{eq:bounded_sre_Rplus}
\limsup_{\ell\to\infty} \frac{\Var(R_+)}{(\Exp R_+)^2} < \infty, \quad 
\lim_{\ell\to\infty} T_{B_\ell}\cdot p_{B_\ell}^{-\delta} = 0.
\end{equation}
\end{itemize}
Then the RMS estimator $\hat\gamma_\ell$ for $\gamma_\ell$, with the optimal choice of the parameters \eqref{eq:optimal_parameters}, is logarithmically efficient.
\end{theorem}
We point out that the first part of the assumption \eqref{eq:bounded_sre_Rplus} is equivalent to strong efficiency of the crude Monte Carlo estimator for $R_+$, under the workload assumption \eqref{ass:workload}. This is not too restrictive, as often the main difficulty \magenta{when estimating} $T_B$ \magenta{lies in} the fact that $p_B$ is extremely small (and \magenta{does not relate to} the large variance of $R_+$.) Since $\gamma_\ell\to 0$ and $A$ is fixed then necessarily $T_{B_\ell}\to0$. In the second part of \eqref{eq:bounded_sre_Rplus} we require that there exists a $\delta>0$ such that $\Exp R_+p_{B_\ell}^{1-\delta}\to0$. Loosely speaking, it means that $p_{B_\ell}$ converges to $0$ at least polynomially faster than $\Exp R_+$ grows to infinity; this is trivially satisfied when $\Exp R_+$ is bounded.

\begin{proof}[Proof of Theorem \ref{thm:logeff_rms}]
Since the recurrency set $A$ is fixed, \magenta{the quantities} $\hat\a_A$, $\RE(\hat\a_A)$ \magenta{and} $W(\hat\a_A)$ do not \magenta{depend} on $\ell$. \magenta{In addition, $\a_A\cdot T_{B_\ell}=\mu(B_\ell)\to0$ is equivalent to $T_{B_\ell}\to 0$.} Moreover, since $T_{B_\ell} = p_{B_\ell}\cdot\Exp R_+$, \magenta{cf.} \eqref{eq:Tb_decomposition}, and $\Exp R_+\geq1$, we necessarily have $p_{B_\ell}\to0$, as $\ell$ grows. \magenta{Observe that}
\begin{align}
\nonumber& \frac{W(\hat\gamma_\ell)\Var(\hat \gamma_\ell)}{\gamma_\ell^{2-\ep}} = \frac{W(\hat\a_A)+W(\hat T_{B_\ell})}{\gamma_\ell^{-\ep}} \cdot \frac{\Var(\hat\a_A\cdot\hat T_{B_\ell})}{\a_A^{2}\cdot T_{B_\ell}^{2}} \\
%& \quad\quad \cdot (\RE(\hat\a_A)+\RE(\hat T_{B_\ell})) \\
\label{step_logeff}& = \gamma_\ell^{\ep}\big(W(\hat\a_A)+W(\hat T_{B_\ell})\big) \cdot \big(\RE(\hat\a_A)+\RE(\hat T_{B_\ell})\big)
%& \sim \a_A^{\ep}p_{B_\ell}^{\ep}(\Exp R_+)^{\ep}W(\hat T_{B_\ell}) \cdot \big(\RE(\hat\a_A)+\RE(\hat T_{B_\ell})\big) \\
%& \sim \a_A^{\ep}p_{B_\ell}^{\ep}(\Exp R_+)^{\ep}W(\hat T_{B_\ell}) \cdot \big(\RE(\hat\a_A)+\RE(\hat T_{B_\ell})\big)
\end{align}
We put $\RE(\hat T_{B_\ell}) = q$. Then the workload $W(\hat T_{B_\ell})$ is given as in \eqref{eq:optimal_workload}, and we see that
\begin{align*}
\gamma_\ell^{\ep}W(\hat T_{B_\ell}) & = \a_A^{\ep}T_{B_\ell}^{\ep} \cdot \frac{1}{q}\bigg(\frac{c\,|\log p_{B_\ell}|}{\sqrt{2c-1}} + \RE(R_+)\bigg)^2 \\
& \sim \frac{c^2\a_A^{\ep} (T_{B_\ell} p_{B_\ell}^{-\delta})^\ep}{q(2c-1)}\cdot p_{B_\ell}^{\delta\ep}(\log p_{B_\ell})^2,
\end{align*}
where $\delta>0$ is as in \eqref{eq:bounded_sre_Rplus}. Now since $p_{B_\ell}\to0$, we also have $$p_{B_\ell}^{\delta\ep}(\log p_{B_\ell})^2\to0,$$ and $\gamma_\ell^{\ep}W(\hat T_{B_\ell}) \to 0$, which applied to \eqref{step_logeff} finishes the proof.
\end{proof}

\medskip
\bibliographystyle{apalike}
\bibliography{bibliografia_recurrent_splitting}
\end{document}